\documentclass[11pt,a4paper]{amsart}
\usepackage[T1]{fontenc}
\usepackage{lmodern}
\usepackage{amsmath,amssymb,amsthm,mathtools}
\usepackage{microtype}
\usepackage[textwidth=155mm,textheight=235mm,centering,headheight=14pt,headsep=7mm]{geometry}
\usepackage{enumitem,needspace,booktabs}
\usepackage{xurl}
\usepackage[hidelinks]{hyperref}
\DeclareMathOperator{\Irr}{Irr}

\DeclareMathOperator{\Bl}{Bl}
\DeclareMathOperator{\bl}{bl}
\DeclareMathOperator{\GL}{GL}

\DeclareMathOperator{\tr}{tr}
\DeclareMathOperator{\Syl}{Syl}
\DeclareMathOperator{\Aut}{Aut}
\DeclareMathOperator{\Inn}{Inn}

\DeclareMathOperator{\Gal}{Gal}

\newtheorem{lemma}{Lemma}[section]
\newtheorem{theorem}[lemma]{Theorem}
\newtheorem{proposition}[lemma]{Proposition}
\newtheorem{corollary}[lemma]{Corollary}
\theoremstyle{definition}
\newtheorem{definition}[lemma]{Definition}
\newtheorem{hypothesis}[lemma]{Hypothesis}

\theoremstyle{remark}
\newtheorem{remark}[lemma]{Remark}
\theoremstyle{plain}
\newtheorem*{maintheoremA}{Theorem A}

\numberwithin{equation}{section}
\setlist[enumerate,1]{label=\textup{(\roman*)},ref=\roman*,leftmargin=2.6em}
\date{}
\title[A reduction theorem for AMN]{A reduction theorem for the
Alperin--McKay--Navarro conjecture}
\author{Shi Chen}
\address{School of Mathematics and Statistics, Central China Normal University,
Wuhan 430079, China}
\email{chenshitjnu@163.com}
\hypersetup{pdftitle={A reduction theorem for the Alperin-McKay-Navarro conjecture},pdfauthor={Shi Chen}}
\begin{document}
\begin{abstract}
We reduce the Alperin--McKay--Navarro conjecture to an inductive condition
on the universal covering groups of non-abelian finite simple groups.
The reduction yields height-zero character bijections compatible with
Brauer correspondence and block relations between $\mathcal H$-triples
for arbitrary finite ambient groups. We record the resulting arithmetic
and structural consequences, including preservation of $p$-rationality
levels and character-theoretic criteria for Sylow subgroups.
The proof combines semilinear
centralization and transfer through quasisimple components with
induction on the central index. It uses the Clifford and gluing theorems
for block relations and a central-defect Dade--Glauberman--Nagao
correspondence established in two related papers. We also verify the
inductive condition for several sporadic groups at primes for which
the Sylow subgroups have prime order, and in defining characteristic
for the Suzuki groups ${}^2B_2(2^{2m+1})$ with $m\geq2$ and the
small Ree groups ${}^2G_2(3^{2m+1})$ with $m\geq1$.
\end{abstract}
\maketitle
\section*{Introduction}

Let $G$ be a finite group, let $p$ be a prime, and let $b$ be a
$p$-block of $G$ with defect group $D$. Write $\Irr_0(b)$ for the
set of irreducible ordinary characters of height zero in $b$.
The Alperin--McKay conjecture \cite{Alp76} predicts that
\[
 |\Irr_0(b)|=|\Irr_0(c)|,
\]
where $c$ is the Brauer correspondent of $b$ in $N_G(D)$.

Reduction theorems have made it possible to approach this conjecture
through the classification of finite simple groups. Following the
reduction of the McKay conjecture by Isaacs, Malle and Navarro
\cite{IMN07}, Sp\"ath \cite[Theorem~C]{Sp13} reduced the
Alperin--McKay conjecture to an inductive condition on the covering
groups of non-abelian finite simple groups. This condition requires
character bijections compatible with the extensions and quotients
that occur in the reduction.

Several families of blocks are now understood in this framework.
Koshitani and Sp\"ath verified the inductive Alperin--McKay condition
for blocks with cyclic defect groups \cite{KS16a,KS16b}. For groups
of type~$A$, Cabanes and Sp\"ath treated blocks of maximal defect
\cite{CS14}, and Brough and Sp\"ath obtained further results for
linear and unitary groups \cite{BS20}. Ruhstorfer subsequently
verified the condition for all blocks of quasisimple groups of
type~$A$ at primes $p\geq5$ \cite[Theorem~C]{Ruh22}. At the
prime~$2$, Brough and Ruhstorfer proved the conjecture for blocks of
maximal defect \cite{BR22}, and Ruhstorfer established it for blocks
with abelian defect groups \cite[Theorem~F]{Ruh22}. These advances
culminated in Ruhstorfer's proof of the Alperin--McKay conjecture for
all finite groups at the prime~$2$ \cite{Ruh25}.

Navarro's refinement \cite{Nav04} also takes character values into
account. Let $\mathbb Q^{\mathrm{ab}}$ be the field generated by all
roots of unity over $\mathbb Q$. For the fixed prime $p$, let
$\mathcal H=\mathcal H_p\leq
\Gal(\mathbb Q^{\mathrm{ab}}/\mathbb Q)$ consist of those
$\sigma$ whose action on roots of unity of order prime to $p$
is an integral power of the automorphism $\xi\mapsto\xi^p$. The Alperin--McKay--Navarro conjecture predicts that,
for every $\sigma\in\mathcal H$,
\[
 |\Irr_0(b)^\sigma|=|\Irr_0(c)^\sigma|,
 \qquad
 \Irr_0(b)^\sigma
   =\{\chi\in\Irr_0(b)\mid\chi^\sigma=\chi\}.
\]

We consider the stronger equivariant formulation, which asks for an
$\mathcal H_b$-equivariant bijection
\[
 \Irr_0(b)\longrightarrow\Irr_0(c),
\]
where $\mathcal H_b$ is the stabilizer of $b$ in $\mathcal H$.
Brauer correspondence commutes with the Galois action, so
$\mathcal H_b=\mathcal H_c$. Such a bijection identifies the fixed
points of every element of this stabilizer. Outside the stabilizer,
both fixed-point sets are empty. Thus this equivariant formulation
implies the stated fixed-point formula; see also
Proposition~\ref{red:fixedpoints} and Corollary~\ref{r7:4}.

This formulation relates the arithmetic of character values to local
block theory. It preserves the $p$-rationality levels of corresponding
height-zero characters, as discussed in \cite[Section~7.1]{HS26}.
The study of fields of values of $p'$-degree characters in
\cite{NT21} provides a broader context for these arithmetic
questions. In particular, the equivariant formulation predicts
local formulas for the numbers of $p$-rational and almost
$p$-rational height-zero characters.

There are also concrete consequences for the structure of Sylow
subgroups. Galois-fixed character counts in principal blocks detect
cyclic Sylow subgroups at the primes $2$ and $3$ \cite{RSV20}, and
the minimum number of generators of a Sylow $2$-subgroup is two
precisely when the appropriate fixed-point count is four
\cite{NRSV21}. These results have independent proofs and are also
consequences of AMN. At the prime $3$, AMN predicts an analogous
criterion with fixed-point count six or nine; one direction was
proved in \cite[Theorem~A]{KMRS26}. Further consequences include
self-normalizing Sylow criteria \cite[Section~5]{Nav04} and the
lower bound for almost $p$-rational characters in principal blocks
proved in \cite[Theorem~A]{MMRSV26}. We record precise statements,
with their hypotheses and attribution, in
Section~\ref{sec:consequences}.

The Galois refinement is known for $p$-solvable groups by work of
Turull \cite{Tur13}, in a stronger form incorporating $p$-adic
fields of values and local Schur indices, and for alternating groups
by Brunat and Nath \cite{BN21}. Further cases include blocks with
cyclic defect groups \cite{Nav04} and blocks with Klein four defect
groups \cite{Hua23}. In non-defining characteristic, Huang, Li and
Zhang proved Turull's refinement for unipotent blocks of general
linear groups with abelian defect groups
\cite[Corollary~1.3]{HLZ24}. The latter results illustrate the role
of descent of equivalences in constructing character correspondences
compatible with Galois actions.

The purpose of this paper is to reduce the equivariant formulation
to an inductive condition, called iAMN, on the universal covering
groups of non-abelian finite simple groups
(Definition~\ref{r6:1}). The conclusion also includes block
relations between $\mathcal H$-triples for arbitrary finite ambient
groups. To state it, for a $p$-subgroup $D\leq X$ write
$\Irr_0(X\mid D)$ for the set of height-zero characters lying in
blocks of $X$ having the specified subgroup $D$ as a defect group.
If $X\trianglelefteq A$, we call $A$ an \emph{ambient group} of
$X$ and write $A_{\chi^{\mathcal H}}$ for the stabilizer of the
Galois orbit $\chi^{\mathcal H}$ in $A$. The block relation
$\geq_b$ is defined in Definition~\ref{red:block}.

\begin{maintheoremA}
Fix a prime $p$. Suppose that every non-abelian finite simple group
satisfies the iAMN condition for $p$. Let $X\trianglelefteq A$, where
$A$ is a finite group, and let $D\leq X$ be a $p$-subgroup. Then there
is an $N_A(D)\times\mathcal H$-equivariant bijection
\[
 \Omega_D:\operatorname{Irr}_0(X\mid D)
 \longrightarrow \operatorname{Irr}_0(N_X(D)\mid D).
\]
For each block $b$ of $X$ with defect group $D$, this bijection maps
$\operatorname{Irr}_0(b)$ onto $\operatorname{Irr}_0(c)$, where $c$ is
the Brauer correspondent of $b$ in $N_X(D)$. Moreover, if
$\chi'=\Omega_D(\chi)$, then
\[
 (A_{\chi^{\mathcal H}},X,\chi)_{\mathcal H}
 \geq_b
 \bigl((N_A(D))_{(\chi')^{\mathcal H}},
        N_X(D),\chi'\bigr)_{\mathcal H}.
\]
In particular, the Alperin--McKay--Navarro conjecture holds for all
finite groups at $p$.
\end{maintheoremA}

The theorem makes the arithmetic and structural consequences of AMN
accessible through an inductive condition on simple groups. The
remaining task is to verify that condition for all such groups at
the chosen prime. The compatibility with ambient groups and block
relations is needed to pass from these verifications to arbitrary
finite groups.

The inductive condition requires an equivariant bijection that
preserves central characters and matches blocks with the same
Brauer correspondent in the relevant normalizer. For each pair of
corresponding characters, it imposes a block relation between
$\mathcal H$-triples after factoring out the kernel of the common
central character. Lemma~\ref{r6:2} gives a criterion for this
relation.

Our framework combines the block character-triple methods of
Navarro and Sp\"ath \cite{NS14} with the theory of
$\mathcal H$-triples used by Navarro, Sp\"ath and Vallejo in their
reduction of the Galois--McKay conjecture \cite{NSV20}.
Character-triple methods also underlie the reduction of Dade's
projective conjecture \cite{Sp17}. Here the required compatibility
with both the mixed group--Galois action and block induction must
be realized by a single pair of associated projective
representations. The same pair determines the factor sets, central
scalars and mixed comparison functions, as well as the character
correspondences on intermediate groups. Under each such
correspondence, the block of the local character must induce to the
block of its global correspondent.

Two related papers provide the transfer results used in the proof.
The Clifford and gluing theorems of \cite{ChenCT} assemble
correspondences above characters of normal subgroups. Their
conclusions give the full block $\mathcal H$-triple relation and
allow the indices of the inducing subgroups on the two sides to
differ. The central-defect Dade--Glauberman--Nagao correspondence
of \cite{ChDGN} allows a nontrivial central intersection with the
defect group and supplies the required relation in that case.
Its integral and multiplicity-algebra constructions use Ladisch's
magic-representation approach \cite{Lad11} and Fu's work on
Galois-equivariant Brauer-character correspondences \cite{Fu26}.
Section~\ref{sec:inputs} records the precise inputs and how they
are used.

The proof proceeds by induction on the central index $|X:Z(X)|$,
simultaneously for all finite ambient groups $A$. This follows the
normal-subgroup reduction strategy in Murai's work
\cite{Mur04,Mur11}. Allowing the ambient group to vary is essential
in the semilinear centralization step: induction is applied to a
group of smaller central index inside an auxiliary finite semidirect
product. The final argument separates three cases. A suitable
normal subgroup first gives descent to a proper subgroup. If this
case does not occur, a central intersection with the defect group
is treated by the Dade--Glauberman--Nagao correspondence, while the
remaining component case is treated by transferring the inductive
condition to direct products and central quotients of universal
covering groups.

We also verify iAMN directly for non-abelian finite simple groups
$S$ with trivial Schur multiplier, trivial outer automorphism group
and $|S|_p=p$. The proof combines the cyclic-defect correspondence
of \cite{Nav04} with explicit representations affording the block
$\mathcal H$-triple relation for the inner-automorphism action.
This yields the cases for $M_{11}$, $J_1$ and $M_{23}$ listed in
Section~\ref{sec:verification}. We also verify iAMN in defining
characteristic for the Suzuki groups ${}^2B_2(2^{2m+1})$ with
$m\geq2$ and the small Ree groups ${}^2G_2(3^{2m+1})$ with
$m\geq1$ (Theorem~\ref{r8:5}). We use the equivariant bijections
and character extensions constructed by Johansson~\cite{Joh22}.
The trace criterion of \cite[Theorem~4.4]{NS14} shows that the
same representations also satisfy the block condition on every
intermediate group.

Section~\ref{sec:inductive} fixes the notation and states the
inductive condition. Section~\ref{sec:inputs} records the Clifford,
gluing and Dade--Glauberman--Nagao results used in the proof.
Section~\ref{sec:centralization} establishes semilinear
centralization and descent, and Section~\ref{sec:components}
transfers the inductive condition through quasisimple components.
These results are combined in the reduction proof in
Section~\ref{sec:reduction}, which also records the arithmetic and
structural consequences. Section~\ref{sec:verification}
contains the simple-group verifications.

\section{Notation and the inductive condition}\label{sec:inductive}
\subsection{Notation and character fibers}
We work throughout with ordinary characters of finite groups. Fix a
prime $p$, and let $\mathcal H=\mathcal H_p$ be the Galois group defined
in the introduction. Although $\mathcal H$ is infinite, its action on
the characters of any fixed finite group factors through a finite
quotient. We use right actions, with the convention
\[
 \chi^{h\sigma}(x)=\sigma\bigl(\chi(hxh^{-1})\bigr),
 \qquad (\chi^a)^b=\chi^{ab}.
\]
For $X\trianglelefteq A$, the group $A$ is called an \emph{ambient group}
of $X$. Statements quantified over all ambient groups allow any finite
group $A$ in which $X$ is normal. Write $A_\chi$ for the inertia group and
\[
 A_{\chi^{\mathcal H}}=\{a\in A\mid\chi^a\in\chi^{\mathcal H}\}.
\]
The superscript $\mathcal H$ denotes a Galois orbit. For conjugation by
a group $T$, we write the orbit as $\{\chi^t\mid t\in T\}$ to distinguish
it from an induced character.

Write $\Bl(X)$ for the set of $p$-blocks of $X$ and $\bl(\chi)$ for
the block containing an irreducible character $\chi$.
For a $p$-block $b$, let $\Irr_0(b)$ denote the set of its height-zero
irreducible characters,
and set $\Irr_0(X)=\bigcup_{b\in\Bl(X)}\Irr_0(b)$.
For a $p$-subgroup $D\leq X$, set
\[
 \Irr_0(X\mid D)=
 \bigcup_{\substack{b\in\Bl(X)\\D\text{ a defect group of }b}}\Irr_0(b).
\]
Here $D$ denotes a specified subgroup. If $K\trianglelefteq X$ and
$\Theta\subseteq\Irr(K)$, set
\[
 \Irr(X\mid\Theta)=\bigcup_{\theta\in\Theta}\Irr(X\mid\theta),
 \qquad
 \Irr_0(X\mid D,\Theta)=\Irr_0(X\mid D)\cap\Irr(X\mid\Theta).
\]
If no defect group is specified, write
$\Irr_0(X\mid\Theta)=\Irr_0(X)\cap\Irr(X\mid\Theta)$.
We omit braces when the indexing set is a singleton. Fibers over
conjugate characters of $K$ coincide, so the distinct fibers are indexed
by $X$-orbits and are pairwise disjoint. If $Z\leq Z(X)$ and
$\lambda\in\Irr(Z)$, the central-character fiber consists of the characters
$\chi$ with $\chi_Z=\chi(1)\lambda$.

We use Brauer's definition of block induction and write $c^J$ only when
the induced block is defined. For modular arguments, we fix a sufficiently
large splitting $p$-modular system $(K,\mathcal O,k)$, with
$k=\mathcal O/J(\mathcal O)$. Here $J(\mathcal O)$ is the Jacobson radical,
which is the unique maximal ideal of the discrete valuation ring
$\mathcal O$, and $a^*$ denotes the residue of $a\in\mathcal O$.
A $p$-subgroup $R\leq U$ is \emph{radical} if $R=O_p(N_U(R))$,
where $O_p(V)$ denotes the largest normal $p$-subgroup of $V$.
The coefficient field $K$ is distinct from any subgroup denoted by the
same letter in a local configuration. We also write
$m(X)=|X:Z(X)|$.

\subsection{Block relations}
We recall the conventions and block relation of~\cite{ChenCT}. A triple
$(G,N,\theta)_{\mathcal H}$ is an
$\mathcal H$-triple if $N\trianglelefteq G$, $\theta\in\Irr(N)$, and
$\{\theta^g\mid g\in G\}\subseteq\theta^{\mathcal H}$.
For a subgroup $T\leq G$, the mixed stabilizer of $\theta$ is
\[
 (T\times\mathcal H)_\theta
 =\{(t,\sigma)\in T\times\mathcal H\mid\theta^{t\sigma}=\theta\}.
\]
An associated projective representation $\mathcal P$ of $G_\theta$
restricts on $N$ to a representation affording $\theta$ and satisfies
\[
 \mathcal P(nx)=\mathcal P(n)\mathcal P(x),\qquad
 \mathcal P(xn)=\mathcal P(x)\mathcal P(n)
 \quad(n\in N,\ x\in G_\theta).
\]
Its factor set is defined by
$\mathcal P(x)\mathcal P(y)=\alpha(x,y)\mathcal P(xy)$.
For $a=(h,\sigma)$ in the mixed stabilizer of $\theta$, the normalized
comparison function $\mu_a$ is determined by
\[
 \mathcal P(hxh^{-1})^\sigma
   =\mu_a(x)T_a\mathcal P(x)T_a^{-1}.
\]
This function is constant on $N$-cosets and takes the value $1$ at the
identity; see~\cite[Lemma~1.4]{NSV20}. In each construction, all
representations are taken over a sufficiently large finite cyclotomic
field, and all factor sets take values in roots of unity.

\begin{definition}\label{red:block}
Let $N\trianglelefteq G$, let $H\leq G$, and set $M=N\cap H$.
Suppose that $\theta\in\Irr(N)$ and $\phi\in\Irr(M)$ satisfy
$\{\theta^g\mid g\in G\}\subseteq\theta^{\mathcal H}$ and
$\{\phi^h\mid h\in H\}\subseteq\phi^{\mathcal H}$. We write
\[
 (G,N,\theta)_{\mathcal H}\geq_b(H,M,\phi)_{\mathcal H}
\]
if the following conditions hold.
\begin{enumerate}
\item $G=NH$, $C_G(N)\leq H$, and
      $(H\times\mathcal H)_\theta=(H\times\mathcal H)_\phi$.
\item There are associated projective representations
      \[
       \mathcal P:G_\theta\longrightarrow\GL_{\theta(1)}
                  (\mathbb Q^{\mathrm{ab}}),\qquad
       \mathcal P':H_\phi\longrightarrow\GL_{\phi(1)}
                  (\mathbb Q^{\mathrm{ab}})
      \]
      whose factor sets take values in roots of unity and agree on
      $H_\phi$. For every $c\in C_G(N)$, the matrices $\mathcal P(c)$
      and $\mathcal P'(c)$ are scalar matrices with the same scalar.
\item For every $a\in(H\times\mathcal H)_\theta$, the comparison
      functions satisfy $\mu'_a=\mu_a|_{H_\phi}$.
\item The blocks $\bl(\theta)$ and $\bl(\phi)$ have a common defect
      group $D$ such that $N_N(D)\leq M$. For every
      $N\leq W\leq G_\theta$, the correspondence
      \[
       \tau_W:\Irr(W\mid\theta)\longrightarrow\Irr(W\cap H\mid\phi)
      \]
      afforded by $(\mathcal P,\mathcal P')$ satisfies
      \[
       \bl(\tau_W(\xi))^W=\bl(\xi)
       \qquad\bigl(\xi\in\Irr(W\mid\theta)\bigr).
      \]
      In particular, the indicated block induction is defined.
\end{enumerate}
\end{definition}

The correspondence $\tau_W$ in the definition is the tensor
correspondence. Explicitly, if $\mathcal Q$ is an irreducible projective representation
of $W/N$
whose inflated factor set cancels that of $\mathcal P|_W$, then
\[
 \tr(\mathcal Q\otimes\mathcal P|_W)
 \longmapsto
 \tr(\mathcal Q|_{W\cap H}\otimes\mathcal P'|_{W\cap H}).
\]
Here the quotient $W/N$ is identified with $(W\cap H)/M$, and the
notation on the right includes inflation of $\mathcal Q$ through this
identification.

The relation $\geq_c$ is defined by conditions \textup{(i)--(iii)} of
Definition~\ref{red:block}. For the relation $\geq_b$, the same pair of
projective representations must also afford every correspondence
required in \textup{(iv)}.

\begin{definition}\label{red:assertion}
For $X\trianglelefteq A$ and a $p$-subgroup $D\leq X$, let
$\mathsf S(A,X,D)$ denote the assertion that there is an
$N_A(D)\times\mathcal H$-equivariant bijection
\[
 \Omega_D^X:\Irr_0(X\mid D)\longrightarrow\Irr_0(N_X(D)\mid D)
\]
satisfying $\bl(\Omega_D^X(\chi))^X=\bl(\chi)$ and, for
$\chi'=\Omega_D^X(\chi)$,
\begin{equation}\label{red:eq:assertion}
 (A_{\chi^{\mathcal H}},X,\chi)_{\mathcal H}
 \geq_b
 \bigl((N_A(D))_{(\chi')^{\mathcal H}},N_X(D),\chi'\bigr)_{\mathcal H}.
\end{equation}
\end{definition}

\begin{proposition}\label{red:fixedpoints}
Suppose that $\mathsf S(G,G,D)$ holds. If $b\in\Bl(G)$ has defect
group $D$ and $b'\in\Bl(N_G(D))$ is its Brauer correspondent, then
there is an $\mathcal H_b$-equivariant bijection
$\Irr_0(b)\longrightarrow\Irr_0(b')$. In particular, for every
$\sigma\in\mathcal H_b$,
\[
 |\Irr_0(b)^\sigma|=|\Irr_0(b')^\sigma|,
\]
where the superscript $\sigma$ denotes the fixed-point set.
\end{proposition}

\begin{proof}
The block condition and the uniqueness assertion in Brauer's first
main theorem imply that $\Omega_D^G$ restricts to the stated bijection.
By \cite[Lemma~1.4]{ChenCT}, every $\sigma\in\mathcal H_b$ also
stabilizes $b'$. Equivariance therefore gives a bijection between the
corresponding $\sigma$-fixed-point sets, proving the equality.
\end{proof}

\subsection{The inductive condition}\label{r6:sub:1}

We impose the inductive condition on the universal covering group of
each non-abelian simple group. The block conditions and the
compatibility conditions for the Galois action must be satisfied by a
single pair of associated projective representations.

\begin{definition}[The iAMN condition]\label{r6:1}
Let $S$ be a non-abelian simple group, and let $U$ be its universal
covering group. We say that $S$ satisfies the \emph{iAMN condition} for
$p$ if the following requirements hold for every non-central radical
$p$-subgroup $R$ of $U$. Set $\Gamma=\Aut(U)_R$. There exist a
$\Gamma$-stable subgroup
\[
 N_U(R)\leq L_R<U
\]
and a $\Gamma\times\mathcal H$-equivariant bijection
\[
 \omega_R:\Irr_0(U\mid R)\longrightarrow\Irr_0(L_R\mid R)
\]
that preserves central characters on $Z(U)$. Moreover, $\omega_R$ maps
the height-zero characters in each block of $U$ represented in its domain
onto those in the block of $L_R$ with the same Brauer correspondent in $N_U(R)$.

For $\vartheta\in\Irr_0(U\mid R)$, set
\[
 \varphi=\omega_R(\vartheta),\qquad
 W=Z(U)\cap\ker\vartheta=Z(U)\cap\ker\varphi,\qquad V=U/W.
\]
Let $\bar\vartheta$ and $\bar\varphi$ be the corresponding characters
of $V$ and $L_R/W$, respectively. We require
\begin{equation}\label{r6:eq:1}
 \bigl(V\rtimes\Gamma_{\vartheta^{\mathcal H}},V,
       \bar\vartheta\bigr)_{\mathcal H}
 \geq_b
 \bigl((L_R/W)\rtimes\Gamma_{\vartheta^{\mathcal H}},L_R/W,
       \bar\varphi\bigr)_{\mathcal H}.
\end{equation}
A single pair of associated projective representations must afford this
relation, satisfying the central scalar condition, equality of all mixed
comparison functions, and the block condition on every intermediate
subgroup.

When $R$ is a central radical $p$-subgroup, we set $L_R=U$ and
$\omega_R=\mathrm{id}$ and use identical projective representations on
the two sides.
\end{definition}

The subgroup $W$ is invariant under $\Gamma_{\vartheta^{\mathcal H}}$,
so the semidirect products in \eqref{r6:eq:1} are defined. Pointwise
invariance of $W$ is not required. Passing to the quotient by the kernel
of the central character follows the ordinary inductive Alperin--McKay
framework of \cite[Definition~2.1(c)]{CS14}. In the present setting,
the mixed Galois comparison functions must also agree for the pair
affording the block relation. Separate pairs affording an ordinary
block relation and a Galois relation do not suffice to establish
\eqref{r6:eq:1}.

\subsection{A criterion for the inductive relation}\label{r6:sub:2}

The following criterion characterizes the required relation on the
central-character quotient. Every condition refers to the same pair of
projective representations. The common defect group and the normalizer
inclusion follow from the local data in Definition~\ref{r6:1}.

\begin{lemma}\label{r6:2}
Fix a non-central radical $p$-subgroup $R$ of $U$. Suppose that $L_R$
and $\omega_R$ satisfy the subgroup, equivariance, central-character,
and blockwise requirements preceding \eqref{r6:eq:1} in
Definition~\ref{r6:1}. For $\vartheta\in\Irr_0(U\mid R)$, set
\[
 \varphi=\omega_R(\vartheta),\qquad
 W=Z(U)\cap\ker\vartheta=Z(U)\cap\ker\varphi,
\]
and write
\[
 \begin{gathered}
 V=U/W,\qquad L=L_R/W,\qquad \bar R=RW/W,\\
 \Gamma_0=\Gamma_{\vartheta^{\mathcal H}},\qquad
 G=V\rtimes\Gamma_0,\qquad H=L\rtimes\Gamma_0.
 \end{gathered}
\]
Let $\bar\vartheta\in\Irr(V)$ and $\bar\varphi\in\Irr(L)$ be the
deflated characters. For a pair $(\mathcal P,\mathcal P')$ of associated
projective representations of $G_{\bar\vartheta}$ and
$H_{\bar\varphi}$, the following are equivalent.
\begin{enumerate}[label=\textup{(\roman*)}]
\item The pair affords \eqref{r6:eq:1}.
\item The pair affords
\[
 (G,V,\bar\vartheta)_{\mathcal H}
 \geq_c(H,L,\bar\varphi)_{\mathcal H},
\]
the specified defect group $\bar R$ satisfies
$C_{G_{\bar\vartheta}}(\bar R)\leq H_{\bar\varphi}$, and, for every
$V\leq J\leq G_{\bar\vartheta}$, the correspondence afforded by this
pair,
\[
 \tau_J:\Irr(J\mid\bar\vartheta)
       \longrightarrow\Irr(J\cap H_{\bar\varphi}\mid\bar\varphi),
\]
satisfies
$\bl(\tau_J(\xi))^J=\bl(\xi)$ for every
$\xi\in\Irr(J\mid\bar\vartheta)$, with the indicated block induction
defined.
\end{enumerate}
In either case, $\bar R$ is a common defect group of
$\bl(\bar\vartheta)$ and $\bl(\bar\varphi)$, and
$N_{G_{\bar\vartheta}}(\bar R)\leq H_{\bar\varphi}$.
\end{lemma}

\begin{proof}
We first establish the assertions about the specified defect group,
without assuming either condition on the projective representations.
The blockwise property of $\omega_R$ implies that $R$ is a common defect
group of $\bl(\vartheta)$ and $\bl(\varphi)$. Since $W\leq Z(U)$ and every
block defect group contains $Z(U)_p$, we have $W_p\leq R$. Hence
\[
 RW=RW_{p'},
\]
and $R$ is the unique Sylow $p$-subgroup of $RW$. Quotient first by the
central $p$-subgroup $W_p$ and then by the image of the central
$p'$-subgroup $W_{p'}$. Both characters have $W$ in their kernels. The
central block correspondences therefore show that $\bar R$ is a common defect group
of $\bl(\bar\vartheta)$ and $\bl(\bar\varphi)$. An element normalizes
$RW$ if and only if it normalizes the unique Sylow $p$-subgroup $R$ of $RW$.
Consequently,
\begin{equation}\label{r6:eq:2}
 N_V(\bar R)=N_U(RW)/W=N_U(R)/W\leq L.
\end{equation}

Assume \textup{(ii)}. The $\geq_c$-relation supplies the group
factorization, equality of mixed stabilizers, common factor sets and
central scalars, and equality of mixed comparison functions. The
preceding paragraph supplies the common defect group and the required
normalizer inclusion. For $J\leq G_{\bar\vartheta}$, equality of
mixed stabilizers gives $J\cap H=J\cap H_{\bar\varphi}$. Thus the
block conditions in \textup{(ii)} are exactly the remaining conditions
of \cite[Definition~1.1]{ChenCT} for the same pair. This proves
\textup{(i)}.

Conversely, assume \textup{(i)}. The $\geq_c$-relation and the block
conditions on every intermediate subgroup are included in
\cite[Definition~1.1]{ChenCT}. It remains to establish the
centralizer inclusion for the specified defect group $\bar R$. Equality
of mixed stabilizers gives
\[
 G_{\bar\vartheta}=V H_{\bar\varphi}.
\]
Since $H_{\bar\varphi}$ stabilizes $\bl(\bar\varphi)$ and the defect
groups of this block are conjugate in $L$, we have
\[
 H_{\bar\varphi}=L N_{H_{\bar\varphi}}(\bar R).
\]
As $L\leq V$, we obtain
$G_{\bar\vartheta}=V N_{H_{\bar\varphi}}(\bar R)$. If
$x\in N_{G_{\bar\vartheta}}(\bar R)$, write $x=vh$ with $v\in V$
and $h\in N_{H_{\bar\varphi}}(\bar R)$. Then
$v\in N_V(\bar R)\leq L$ by \eqref{r6:eq:2}. Therefore
\[
 N_{G_{\bar\vartheta}}(\bar R)
   =N_V(\bar R)N_{H_{\bar\varphi}}(\bar R)
   \leq H_{\bar\varphi}.
\]
In particular, $C_{G_{\bar\vartheta}}(\bar R)\leq H_{\bar\varphi}$,
as required in \textup{(ii)}. This Frattini argument also applies under
\textup{(ii)}: it uses only the $\geq_c$-relation and the local
properties established above. The final assertion follows as well.
\end{proof}

This criterion uses the subgroup and blockwise hypotheses of
Definition~\ref{r6:1}. The equivalence is not asserted for
arbitrary $\mathcal H$-triples without the common-defect and normalizer
hypotheses. When the height-zero sets are empty, no character pairs
remain for which the relation must be checked.

\section{Transfer results and the DGN correspondence}\label{sec:inputs}
\subsection{Character-triple transfer}
We recall the Clifford and gluing theorems from~\cite{ChenCT} used
in the reduction. The block relations in these results are understood
in the sense of Definition~\ref{red:block}. Thus a single pair of
associated projective representations affords both the mixed comparison
functions and the block correspondences for all intermediate groups.

The homogeneous transfer theorem, \cite[Theorem~2.7]{ChenCT}, transfers
a relation between $(G,N,\theta)_{\mathcal H}$ and
$(H,M,\phi)_{\mathcal H}$ to a normal intermediate subgroup contained
in $G_\theta$. The full Clifford transfer theorem,
\cite[Theorem~2.19]{ChenCT}, applies to an arbitrary normal intermediate
subgroup $N\leq J\trianglelefteq G$ and yields both the induced block
relation and equality of relative heights. Its specialization to
height-zero characters with a specified defect group is stated in
\cite[Corollary~2.20]{ChenCT}.

We also use two gluing results. The normal-subgroup transfer theorem,
\cite[Theorem~3.2]{ChenCT}, lifts a strong correspondence on a normal
subgroup to a larger normal subgroup while preserving the specified
upper defect group. The inertia-group gluing theorem,
\cite[Theorem~3.3]{ChenCT}, assembles compatible correspondences on
the inertia groups of characters of a normal subgroup into a
correspondence on the whole group. Its proof uses
\cite[Lemma~3.1]{ChenCT}, which does not require equality between
the indices of the subgroups from which the characters are induced.

The semilinear centralization argument supplies the correspondences
needed for inertia-group gluing. The DGN and component correspondences
provide the input to normal-subgroup transfer. At each application, we
specify the relevant groups and verify the hypotheses.

\subsection{The central-defect configuration}\label{r3:sub:1}
Let $K\leq M$ be normal subgroups of a finite group $A$, with $M/K$ a
$p$-group. Let $b\in\Bl(K)$ have defect group $Z\leq Z(M)$, and choose an
$M$-invariant character $\theta\in\Irr(b)$ such that
\[
 \{\theta^a\mid a\in A\}\subseteq\theta^{\mathcal H}.
\]
Let $B$ be the unique block of $M$ covering $b$, and let $D$ be a defect
group of $B$. Set
\[
 H=N_A(D),\qquad L=N_K(D),\qquad M'=N_M(D).
\]
Write $\phi\in\Irr(L)$ for the generalized DGN correspondent of $\theta$.
The subgroup $Z$ need not lie in $\ker\theta$ or be central in $A$.
The construction therefore keeps track of the central-character fibers.

\subsection{The central-defect DGN theorem}\label{r3:sub:2}
\begin{theorem}[{\cite[Theorem~A]{ChDGN}}]\label{r3:1}
In the preceding configuration, there is an
$H\times\mathcal H$-equivariant bijection
\[
 \Delta_D:\Irr_0(M\mid\theta^{\mathcal H})
       \longrightarrow\Irr_0(M'\mid\phi^{\mathcal H}).
\]
For each $\gamma\in\mathcal H$, the bijection maps the fiber over
$\theta^\gamma$ onto the fiber over $\phi^\gamma$. It matches
Brauer corresponding blocks. If $\eta=\Delta_D(\xi)$, then
\[
 (A_{\xi^{\mathcal H}},M,\xi)_{\mathcal H}
 \geq_b(H_{\eta^{\mathcal H}},M',\eta)_{\mathcal H}.
\]
A single pair of associated projective representations affords the common
factor sets, equal central scalars, mixed comparison functions and block
correspondences for all intermediate groups.
\end{theorem}

The theorem is the central-defect form of the correspondence for
ordinary characters proved in the related paper~\cite{ChDGN}. Its proof
constructs a graded correspondence over the group algebra of $Z$ and
identifies the associated projective representations on each
central-character fiber.
The comparison-function calculation is given in
\cite[Proposition~5.3]{ChDGN}, and the intermediate-group block
identities are proved in \cite[Theorem~6.3]{ChDGN}.
The integral construction and its proof are not repeated here.

When $Z=1$, the group identities reduce to
\[
 K\cap D=1,\qquad N_K(D)=C_K(D),\qquad N_M(D)=D\times C_K(D).
\]
The reduction requires the full central-defect statement, allowing
$K\cap D$ to be nontrivial.

\subsection{Correspondence over all character fibers}\label{r3:sub:3}
The following corollary, proved in \cite[Corollary~7.4]{ChDGN},
combines the correspondences over the relevant characters of $K$.
We recall its proof to explain its role in the induction.

\begin{corollary}\label{r3:2}
Let $K\leq V$ be normal subgroups of a finite group $B$, and let
$D\leq V$ be a $p$-subgroup. Suppose that
\[
 V=KD,\qquad K\cap D\leq Z(V).
\]
There is an $N_B(D)\times\mathcal H$-equivariant bijection
\[
 \Delta:\Irr_0(V\mid D)\longrightarrow\Irr_0(N_V(D)\mid D)
\]
which matches Brauer corresponding blocks. For every pair
$\xi'=\Delta(\xi)$, we have
\[
 (B_{\xi^{\mathcal H}},V,\xi)_{\mathcal H}
 \geq_b\bigl((N_B(D))_{(\xi')^{\mathcal H}},N_V(D),\xi'\bigr)_{\mathcal H}.
\]
\end{corollary}
\begin{proof}
Set $Z=K\cap D$ and take $\xi\in\Irr_0(V\mid D)$.
By~\cite[Proposition~2.5]{NS14}, we may choose a constituent
$\theta\in\Irr_0(K\mid Z)$ of $\xi_K$ such that the Clifford
correspondent of $\xi$ over $\theta$ has defect group $D$. Then
\[
 \theta(1)_p=|K:Z|_p=|V:D|_p=\xi(1)_p.
\]
Since $V/K$ is a $p$-group, the inertia index and the multiplicity
in the Clifford decomposition are powers of $p$. The displayed
equality forces both to be one, so $\xi_K=\theta$.
Apply Theorem~\ref{r3:1} in $B_{\theta^{\mathcal H}}$.
The generalized DGN correspondence matches the characters occurring
on the two normal subgroups. We transport the resulting fiber
correspondences along their $N_B(D)\times\mathcal H$-orbits. This
is well-defined because each fiber correspondence is equivariant
under the full mixed stabilizer of the character indexing the fiber.
Finally, $\xi_K=\theta$ implies
$B_{\xi^{\mathcal H}}\leq B_{\theta^{\mathcal H}}$. Thus the strong
relation is defined on the full ambient group required for $\xi$.
\end{proof}

In Section~\ref{sec:reduction}, we apply this corollary to a normal
section $V=KD$ and then use \cite[Theorem~3.2]{ChenCT} to transfer
the correspondence from $V$ to the required larger normal subgroup.

\section{Semilinear centralization and descent}\label{sec:centralization}

This section establishes the centralization step in the induction on the
central index. We formulate the induction hypothesis for arbitrary finite
overgroups, since the construction replaces the given overgroup by a
finite semidirect product.

\subsection{The induction hypothesis}\label{r4:sub:1}

Fix the prime $p$ and the group
$\mathcal H\leq\operatorname{Gal}(\mathbb Q^{\mathrm{ab}}/\mathbb Q)$
used throughout the paper. We use the facts that $\mathcal H$ is abelian and that
its action preserves roots of unity, character heights, and the specified
defect subgroups. We use the action convention
\[
 \chi^{(h,\sigma)}(x)=\sigma\bigl(\chi(hxh^{-1})\bigr),
 \qquad (\chi^a)^b=\chi^{ab}.
\]
Write $\Irr_0(V\mid Q)$ for the height-zero characters in blocks
with defect group $Q$, and set $m(V)=|V:Z(V)|$. For $K\unlhd V$
and $S\subseteq\Irr(K)$, define
\[
 \Irr_0(V\mid Q,S)=\Irr_0(V\mid Q)\cap
       \bigcup_{\eta\in S}\Irr(V\mid\eta).
\]
We omit the braces when $S$ is a singleton. As in Section~1, all
representations in a given construction are taken over a sufficiently
large finite cyclotomic field, with root-of-unity-valued factor sets.
Throughout this section, $\geq_b$ denotes the full relation of \cite[Definition~1.1]{ChenCT}.
In particular, the same pair of associated projective representations
must afford both the mixed comparison functions and the block
correspondences for all intermediate groups.

\begin{hypothesis}[Induction on the central index]\label{r4:1}
Fix a positive integer $n$. For every finite group $V$ with $m(V)<n$,
every finite group $B$ containing $V$ as a normal subgroup, and every
$p$-subgroup $Q\leq V$, there is an $N_B(Q)\times\mathcal H$-equivariant
bijection
\[
 \Omega_Q:\Irr_0(V\mid Q)\longrightarrow
              \Irr_0(N_V(Q)\mid Q)
\]
such that
\[
 \bl\bigl(\Omega_Q(\rho)\bigr)^V=\bl(\rho).
\]
If $\rho'=\Omega_Q(\rho)$, then
\[
 (B_{\rho^{\mathcal H}},V,\rho)_{\mathcal H}
 \geq_b
 \bigl(N_B(Q)_{(\rho')^{\mathcal H}},N_V(Q),\rho'\bigr)_{\mathcal H}.
\]
In particular, a single pair of associated projective representations
affords the mixed comparison functions and the block correspondences
for all intermediate groups.
\end{hypothesis}

The groups $V$ may be restricted to a class closed under the subgroup,
quotient, and finite central-extension constructions used below. The
hypothesis must still hold for every finite overgroup $B$ in which $V$
is normal.

\subsection{A finite semilinear extension}\label{r4:sub:2}

Let $K\unlhd A$ and $\zeta\in\Irr(K)$, and set
\[
 B=A_\zeta,\qquad A_0=A_{\zeta^{\mathcal H}},\qquad
 \Gamma=(A_0\times\mathcal H)_\zeta.
\]
The group $\Gamma$ need not be finite. Since the group and Galois
actions commute, $B\unlhd A_0$: if $h\in A_0$ and
$\zeta^h=\zeta^\tau$, then
$B^h=A_{\zeta^h}=A_{\zeta^\tau}=B$.

Choose a projective representation
\[
 \mathcal P:B\longrightarrow\operatorname{GL}_d(\mathbb Q^{\mathrm{ab}}),
 \qquad d=\zeta(1),
\]
associated with $(B,K,\zeta)$, whose normalized factor set $\alpha$
takes values in roots of unity and is inflated from $B/K$. Thus
\begin{equation}\label{r4:eq:1}
 \mathcal P(g)\mathcal P(t)=\alpha(g,t)\mathcal P(gt),
 \qquad
 \alpha(k,g)=\alpha(g,k)=1
 \quad(k\in K,\ g\in B).
\end{equation}
The restriction $\mathcal P_K$ is an ordinary irreducible representation
affording $\zeta$. The existence of such a projective representation
follows from \cite{NSV20}. For each $a=(h,\sigma)\in\Gamma$, the same
results give a unique normalized comparison function
$\mu_a:B\longrightarrow(\mathbb Q^{\mathrm{ab}})^\times$, constant on
$K$-cosets and equal to $1$ on $K$, and an invertible matrix $S_a$ such
that
\begin{equation}\label{r4:eq:2}
 \mathcal P(hgh^{-1})^\sigma
   =\mu_a(g)S_a^{-1}\mathcal P(g)S_a
 \qquad(g\in B).
\end{equation}
The construction below uses the comparison functions $\mu_a$ and
requires no multiplicative compatibility among the matrices $S_a$.

\begin{lemma}\label{r4:2}
There exists a finite Galois-stable group $Z$ of roots of unity
containing all values of $\alpha$ and of every comparison function
$\mu_a$, $a\in\Gamma$. For such a group $Z$, define multiplication
on $\widehat B=B\times Z$ and a map $\pi$ by
\begin{equation}\label{r4:eq:3}
 (g,z)(t,w)=(gt,zw\alpha(g,t)),
 \qquad
 \pi:\widehat B\longrightarrow B,\quad(g,z)\longmapsto g.
\end{equation}
Then $\widehat B$ is a finite group with central kernel
$\ker\pi\cong Z$, and
\[
 K_0=\{(k,1):k\in K\}\unlhd\widehat B.
\]
The map $L(g,z)=z\mathcal P(g)$ is an ordinary irreducible
representation of $\widehat B$. Its character
$\widetilde\zeta=\operatorname{tr}L$ extends the character of $K_0$
corresponding to $\zeta$.

For $a=(h,\sigma)\in\Gamma$, define
\begin{equation}\label{r4:eq:4}
 f_a(g,z)=\left(hgh^{-1},\bigl(z\mu_a(g)^{-1}\bigr)^{\sigma^{-1}}\right).
\end{equation}
Then $f_a\in\operatorname{Aut}(\widehat B)$, and
\begin{align}
 f_{ab}&=f_a\circ f_b,\label{r4:eq:5}\\
 f_a(k,1)&=(hkh^{-1},1),\label{r4:eq:6}\\
 f_a(1,z)&=(1,z^{\sigma^{-1}}),\label{r4:eq:7}\\
 L(f_a(g,z))^\sigma&=S_a^{-1}L(g,z)S_a.\label{r4:eq:8}
\end{align}
Thus $a\mapsto f_a$ defines a left action of $\Gamma$ on
$\widehat B$ by automorphisms.
\end{lemma}

\begin{proof}
Comparing products in \eqref{r4:eq:2} gives
\begin{equation}\label{r4:eq:9}
 \alpha(hgh^{-1},hth^{-1})^\sigma
   =\mu_a(g)\mu_a(t)\mu_a(gt)^{-1}\alpha(g,t).
\end{equation}
Let $a=(h,\sigma)$ and $b=(t,\tau)$ belong to $\Gamma$. We have
$(\mathcal P^a)^b=\mathcal P^{ab}$. Applying $b$ to
\eqref{r4:eq:2} gives the intertwiner $S_bS_a^\tau$ and the scalar
function $\mu_a^b\mu_b$. After restriction to $K$, this intertwiner
and $S_{ab}$ intertwine the same pair of irreducible representations.
They therefore differ by a scalar, by Schur's lemma. Substitution in
\eqref{r4:eq:2} for arbitrary elements of $B$ now gives
\begin{equation}\label{r4:eq:10}
 \mu_{ab}(g)=\tau\bigl(\mu_a(tgt^{-1})\bigr)\mu_b(g).
\end{equation}
Thus no additional scalar $2$-cocycle occurs in this identity.

For $g\in B$, equation~\eqref{r4:eq:1} shows that
$\mathcal P(g)^{|g|}$ is a scalar matrix with a root of unity as its
scalar. Hence $\det\mathcal P(g)$ is a root of unity. Since $B$ is
finite, choose a positive integer $N$ divisible by the orders of all
these determinants, and a positive integer $m_\alpha$ divisible by the
orders of all values of $\alpha$.
Taking determinants in \eqref{r4:eq:2} gives
\[
 \mu_a(g)^d=
 \frac{\det\mathcal P(hgh^{-1})^\sigma}{\det\mathcal P(g)}
 \in\mu_N.
\]
It follows that $\mu_a(g)\in\mu_{dN}$ for every $a\in\Gamma$ and
$g\in B$, where $\mu_r$ denotes the group of all $r$th roots of unity
in $\mathbb Q^{\mathrm{ab}}$. We may therefore take
\begin{equation}\label{r4:eq:11}
 m=\operatorname{lcm}(m_\alpha,dN),\qquad Z=\mu_m.
\end{equation}
The group $Z$ is stable under all Galois automorphisms. The bound on
the orders of the values of $\mu_a$ is uniform in $a$; no finiteness
assumption on $\Gamma$ or on the collection of comparison functions
is needed.

The normalized cocycle identity shows that \eqref{r4:eq:3} defines a
group with identity $(1,1)$ and central subgroup
$\{(1,z):z\in Z\}$. Since $\alpha$ is inflated from $B/K$, we have
$K_0\cong K$ and
\[
 (g,z)(k,1)(g,z)^{-1}=(gkg^{-1},1).
\]
Thus $K_0$ is normal. Moreover,
\[
 L(g,z)L(t,w)=zw\alpha(g,t)\mathcal P(gt)
            =L\bigl((g,z)(t,w)\bigr).
\]
Thus $L$ is an ordinary representation. Since its restriction to
$K_0$ is irreducible, $L$ is irreducible and affords the required
extension.

By \eqref{r4:eq:9}, the second coordinate of $f_a(g,z)f_a(t,w)$ is
\begin{align*}
 &\bigl(zw\mu_a(g)^{-1}\mu_a(t)^{-1}\bigr)^{\sigma^{-1}}
          \alpha(hgh^{-1},hth^{-1})\\
 &\hspace{35mm}=
       \bigl(zw\alpha(g,t)\mu_a(gt)^{-1}\bigr)^{\sigma^{-1}}.
\end{align*}
The first coordinate is $hgth^{-1}$, so $f_a$ is a homomorphism.
For the composition law, the first coordinate of $f_a(f_b(g,z))$ is
$htg(ht)^{-1}$, and the second coordinate is
\begin{align*}
 \left[\bigl(z\mu_b(g)^{-1}\bigr)^{\tau^{-1}}
                       \mu_a(tgt^{-1})^{-1}\right]^{\sigma^{-1}}
 &=\left[z\mu_b(g)^{-1}
               \tau\bigl(\mu_a(tgt^{-1})\bigr)^{-1}\right]^{(\sigma\tau)^{-1}}\\
 &=\bigl(z\mu_{ab}(g)^{-1}\bigr)^{(\sigma\tau)^{-1}}.
\end{align*}
Here we used \eqref{r4:eq:10} and the commutativity of the Galois group.
This proves \eqref{r4:eq:5}. Uniqueness of the normalized comparison
function gives $\mu_1=1$. Consequently, $f_1=1$ and $f_{a^{-1}}$ is
the inverse of $f_a$, proving that each $f_a$ is an automorphism.

Equations~\eqref{r4:eq:6} and~\eqref{r4:eq:7} follow from
$\mu_a|_K=1$. Finally,
\[
 L(f_a(g,z))^\sigma
   =z\mu_a(g)^{-1}\mathcal P(hgh^{-1})^\sigma
   =S_a^{-1}\bigl(z\mathcal P(g)\bigr)S_a,
\]
which proves \eqref{r4:eq:8}.
\end{proof}

\begin{remark}\label{r4:3}
For $h\in B$, the automorphism $f_{(h,1)}$ is conjugation by
$(h,1)\in\widehat B$. Indeed, projective multiplication gives
\[
 \mu_{(h,1)}(g)^{-1}
   =\frac{\alpha(h,g)\alpha(hg,h^{-1})}{\alpha(h,h^{-1})},
\]
which is the factor multiplying $z$ in
$(h,1)(g,z)(h,1)^{-1}$. Thus the mixed action restricts to ordinary
conjugation by lifts of elements of $B$.
\end{remark}

\begin{proposition}\label{r4:4}
Suppose that $K\leq X\unlhd A$ and $X\leq B$. Set
\[
 \widehat X=\pi^{-1}(X),\qquad
 \overline B=\widehat B/K_0,\qquad
 V=\widehat X/K_0,\qquad
 F=\operatorname{im}\bigl(\Gamma\longrightarrow
                       \operatorname{Aut}(\widehat B)\bigr).
\]
Then $F$ is finite and stabilizes both $K_0$ and $\widehat X$. It
therefore induces an action on $\overline B$. In particular,
\[
 \mathcal A=\overline B\rtimes F
\]
is a finite overgroup containing $\overline B$, and $V\unlhd\mathcal A$.
If $\overline Z$ is the image of $Z$ in $\overline B$, then
\[
 \overline Z\cong Z,\qquad
 \overline Z\leq Z(V),\qquad
 V/\overline Z\cong X/K,\qquad
 |V:Z(V)|\leq|X:K|.
\]
\end{proposition}

\begin{proof}
Since $\widehat B$ is finite, so is $F$. Lemma~\ref{r4:2} shows that
$F$ stabilizes $K_0$. Each $f_a$ covers conjugation by an element of
$A$ and hence stabilizes $\widehat X$, since $X\unlhd A$. Hence $F$ acts on
$\overline B$. The subgroup $V$ is normal in $\overline B$ and
$F$-stable, so $V\unlhd\mathcal A$. Finally, $Z\cap K_0=1$ and
$Z\leq Z(\widehat B)$ yield the assertions about $\overline Z$,
the quotient isomorphism, and the central-index bound.
\end{proof}

\begin{proposition}\label{r4:5}
With the notation of Proposition~\ref{r4:4}, let
$q:\widehat X\longrightarrow V$ be the quotient map. For every
$\chi\in\Irr(X\mid\zeta)$, its inflation
$\widehat\chi=\chi\circ\pi|_{\widehat X}$ can be written uniquely as
\begin{equation}\label{r4:eq:12}
 \widehat\chi=\widetilde\zeta_{\widehat X}
                \operatorname{Inf}^{\widehat X}_V(\rho),
 \qquad \rho\in\Irr(V).
\end{equation}
If $a=(h,\sigma)\in\Gamma$ and $\overline f_a$ is its induced
automorphism on $V$, then
\begin{equation}\label{r4:eq:13}
 \chi^a=\chi\quad\Longleftrightarrow\quad
       \rho^{(\overline f_a,\sigma)}=\rho,
 \qquad
 \rho^{(\overline f_a,\sigma)}(v)
       =\sigma\bigl(\rho(\overline f_a(v))\bigr).
\end{equation}
\end{proposition}

\begin{proof}
The character $\widetilde\zeta_{\widehat X}$ extends the character on
$K_0$ corresponding to $\zeta$. Gallagher's theorem therefore gives
\eqref{r4:eq:12} and its uniqueness. Equation~\eqref{r4:eq:8} gives
\[
 \widetilde\zeta_{\widehat X}^{(f_a,\sigma)}
       =\widetilde\zeta_{\widehat X}.
\]
Since $q\circ f_a=\overline f_a\circ q$, it follows that
\begin{equation}\label{r4:eq:14}
 \bigl(\widetilde\zeta_{\widehat X}\operatorname{Inf}(\rho)\bigr)^{(f_a,\sigma)}
   =\widetilde\zeta_{\widehat X}
                \operatorname{Inf}\bigl(\rho^{(\overline f_a,\sigma)}\bigr).
\end{equation}
Since $f_a$ covers $g\mapsto hgh^{-1}$, the left-hand side of
\eqref{r4:eq:14}, with $\rho$ as in \eqref{r4:eq:12}, is the inflation
of $\chi^a$. The injectivity of the Gallagher correspondence proves
\eqref{r4:eq:13}.

The corresponding central character is preserved as well. Since
$Z\leq\ker\widehat\chi$ and $L(1,z)=zI_d$, we have
\[
 \rho_{\overline Z}=\rho(1)\nu^{-1},\qquad \nu(z)=z,
\]
where $\overline Z$ is identified with $Z$. This character is fixed by
the corresponding mixed action, because
$\sigma((z^{\sigma^{-1}})^{-1})=z^{-1}$. Hence the stabilizer
correspondence preserves the central fiber defined by the possibly
nontrivial character $\nu^{-1}$.
\end{proof}

\begin{remark}\label{r4:6}
The automorphism $f_a$ need not determine the Galois component
$\sigma$: distinct pairs $a=(h,\sigma)$ can induce the same
automorphism. We therefore retain the pair $(f_a,\sigma)$ for each
$a$. Conjugation by $(1,f_a)$ in
$\mathcal A=\overline B\rtimes F$ induces $\overline f_a$ on $V$.
Thus Hypothesis~\ref{r4:1} applies to this overgroup action together
with the Galois action. Equation~\eqref{r4:eq:13} identifies the
corresponding stabilizer conditions; it does not assert that every
Gallagher factor is fixed by all of $\Gamma$. The construction also
requires no ordinary extension of $\zeta$ to $A_0$.
\end{remark}

\subsection{The central index bound}\label{r4:sub:3}

For the remainder of this section, assume Hypothesis~\ref{r4:1}. Fix
normal subgroups $K\leq X$ of a finite group $A$, a $p$-subgroup
$D\leq X$, and
\[
 D_0=K\cap D,\qquad \zeta\in\Irr_0(K\mid D_0),\qquad
 \zeta\text{ is $X$-invariant},\qquad |X:K|<n.
\]
Set
\[
 A_0=A_{\zeta^{\mathcal H}},\qquad C=N_{A_0}(D),\qquad
 Y=KN_X(D),\qquad E=KC.
\]
We first construct a quotient group to which the induction hypothesis
applies. At this stage we work over the fixed character $\zeta$; its full
$\mathcal H$-orbit is treated in Theorem~\ref{r4:16}. If both
$\Irr_0(X\mid D,\zeta)$ and $\Irr_0(Y\mid D,\zeta)$ are empty, the
empty bijection suffices. We may therefore assume that at least one of
these sets is nonempty. The construction below will also establish the
nonemptiness of the other.

Consider the central extension constructed in Lemma~\ref{r4:2},
\[
 \pi:\widehat B\longrightarrow B=A_\zeta,\qquad
 \ker\pi=Z,\qquad K_0\unlhd\widehat B,
\]
together with its ordinary representation $L$ and extension character
$\widetilde\zeta=\operatorname{tr}L$. Write
\[
 \widehat X=\pi^{-1}(X),\qquad \widehat Y=\pi^{-1}(Y).
\]
The group $\pi^{-1}(D)$ is a central cyclic extension of a $p$-group.
It has a unique Sylow $p$-subgroup $\widehat D$, and
\begin{equation}\label{r4:eq:15}
 \pi^{-1}(D)=\widehat D\times Z_{p'},\qquad
 |\widehat D|=|Z|_p|D|.
\end{equation}
Indeed, $Z_{p'}$ is a central normal Hall $p'$-subgroup.
The Schur--Zassenhaus theorem gives a complement, and the conjugacy of
complements, together with the centrality of $Z_{p'}$, gives uniqueness. Moreover,
\[
 \pi|_{\widehat D\cap K_0}:\widehat D\cap K_0
           \xrightarrow{\ \sim\ }D_0.
\]
Let $\zeta_0$ correspond to $\zeta$ under $K_0\cong K$, and set
$D_{00}=\widehat D\cap K_0$. This subgroup is a defect group of the block containing
$\zeta_0$.

Set $U=K_0\widehat D$ and $b_0=\bl(\zeta_0)$. The nonemptiness assumption for the characters over
$\zeta$ yields a block $b$ of
$\widehat X$ or $\widehat Y$ covering $b_0$ with defect group
$\widehat D$. Indeed, inflate a character from either nonempty set along
$\pi$. The defect formula for a central quotient gives
\eqref{r4:eq:15}; since inflation leaves the degree unchanged, the
inflated character has height zero.

The block $b_0$ is fixed by $\widehat B$. Since $U/K_0$ is a
$p$-group, there is a unique block of $U$ covering $b_0$, and its
idempotent is $b_0$, viewed as an element of $kU$. This covering block contains
$\widetilde\zeta_U$, whose restriction to $K_0$ is $\zeta_0$. By the
defect-intersection theorem and the defect projection formula for an
extension, a defect group $R$ of this block satisfies
\[
 |R\cap K_0|=|D_{00}|,\qquad RK_0/K_0=U/K_0.
\]
Thus $|R|=|\widehat D|$. Since $bb_0=b$ and
$\operatorname{Br}_{\widehat D}(b)\ne0$, we have
\[
 \operatorname{Br}_{\widehat D}(b_0)\ne0.
\]
Thus $\widehat D$ is contained in a defect group of the unique
covering block. Comparing orders, we conclude that $\widehat D$ is itself
a defect group of that block.

Every $g\in N_{\widehat B}(U)$ fixes $b_0$, so the defect groups
$\widehat D^g$ and $\widehat D$ are conjugate in $U$. The block
Frattini argument therefore gives
\begin{equation}\label{r4:eq:16}
 N_{\widehat B}(U)=K_0N_{\widehat B}(\widehat D).
\end{equation}
By \eqref{r4:eq:15} and the characteristic property of the unique
Sylow subgroup,
\[
 N_{\widehat B}(\widehat D)=\pi^{-1}(N_B(D)).
\]
Intersecting with $\widehat X$ gives
\begin{equation}\label{r4:eq:17}
 N_{\widehat X}(U)=K_0N_{\widehat X}(\widehat D)
                  =\pi^{-1}(KN_X(D))=\widehat Y.
\end{equation}
Also $E\cap B=KN_B(D)$, and hence
\begin{equation}\label{r4:eq:18}
 \widehat E:=N_{\widehat B}(U)=\pi^{-1}(E\cap B).
\end{equation}
The nonemptiness assumption is essential to this block argument.
In particular, equations~\eqref{r4:eq:16}--\eqref{r4:eq:18} are not being
deduced from the defect condition on $D_0$ alone.

Now set
\[
 \overline B=\widehat B/K_0,\qquad V=\widehat X/K_0,\qquad
 Q=U/K_0,\qquad\overline Z=K_0Z/K_0.
\]
Since $\overline Z\leq Z(V)$ and $V/\overline Z\cong X/K$, we have
\begin{equation}\label{r4:eq:19}
 m(V)\leq|V:\overline Z|=|X:K|<n.
\end{equation}
Let
$F=\operatorname{im}(\Gamma\longrightarrow\operatorname{Aut}(\widehat B))$
and use the finite overgroup
\begin{equation}\label{r4:eq:20}
 \mathcal A=\overline B\rtimes F,\qquad V\unlhd\mathcal A.
\end{equation}
Thus the quotient overgroup $\overline B$ embeds in
$\mathcal A$. By \eqref{r4:eq:17}, we have
$\widehat Y/K_0=N_V(Q)$. Hypothesis~\ref{r4:1} now yields a
bijection
\begin{equation}\label{r4:eq:21}
 \Omega_Q:\Irr_0(V\mid Q)\longrightarrow\Irr_0(N_V(Q)\mid Q).
\end{equation}

\subsection{Height-zero characters and prescribed defects}\label{r4:sub:4}

We next establish the multiplication result needed for descent.
Inflation need not preserve height zero, so quotient characters must be
distinguished from their inflations. We prove both directions of the required
assertion using \cite[Lemma~2.2, Proposition~2.5, and Theorem~4.4]{NS14}.
The argument below keeps the height-zero hypotheses on quotient
characters separate from the hypothesis on inflated characters in
\cite[Theorem~4.6]{NS14}.

Work over a sufficiently large splitting $p$-modular system. For
$\chi\in\Irr(A)$, let $\lambda_\chi$ denote the central form of its block on
$Z(kA)$. Thus
\[
 \lambda_\chi\bigl(\operatorname{Cl}_A(x)^+\bigr)
 =\left(\frac{|\operatorname{Cl}_A(x)|\chi(x)}{\chi(1)}\right)^*.
\]
Here ${}^*$ denotes reduction of local integers, and ${}^+$ denotes the sum
of the elements of a subset in the group algebra. The displayed fraction is
an algebraic integer, and the form depends only on $\bl(\chi)$.

\begin{lemma}\label{r4:7}
Suppose that $D\leq M\leq N\leq A$, where $D$ is a $p$-group and
$N_N(D)\leq M$. If $x\in C_A(D)$, then
\[
 D\in\Syl_p(C_N(x))\quad\Longleftrightarrow\quad
 D\in\Syl_p(C_M(x)).
\]
If $K\unlhd A$ and $K\leq M$, the same assertion holds for $N/K$ and $M/K$
whenever the corresponding normalizer containment holds. The element $x$ is allowed to lie in the overgroup.
\end{lemma}

\begin{proof}
The forward implication follows from the subgroup inclusion. Conversely,
suppose that $D$ is Sylow in $C_M(x)$
but not in $C_N(x)$, and choose $P\in\Syl_p(C_N(x))$ containing $D$.
Since $D<P$, the normalizer condition for finite $p$-groups gives
\[
 D<N_P(D)\leq N_N(D)\cap C_N(x)\leq C_M(x),
\]
a contradiction. The same argument in the quotient proves the assertion
for $N/K$ and $M/K$.
\end{proof}

\begin{lemma}\label{r4:8}
The following assertions hold.
\begin{enumerate}[label=\textup{(\roman*)}]
\item Let $K\unlhd A$, let $\zeta\in\Irr(A)$ with
$\zeta_K\in\Irr(K)$, and let $\nu$ be the inflation of
$\bar\nu\in\Irr(A/K)$. For $x\in A$, set $\bar x=xK$ and
$L/K=C_{A/K}(\bar x)$. Then
\begin{equation}\label{r4:eq:22}
 \lambda_{\zeta\nu}\bigl(\operatorname{Cl}_A(x)^+\bigr)
 =\lambda_{\zeta_L}\bigl(\operatorname{Cl}_L(x)^+\bigr)
  \lambda_{\bar\nu}\bigl(\operatorname{Cl}_{A/K}(\bar x)^+\bigr).
\end{equation}
\item Suppose that $D\leq M\leq N$ and $N_N(D)\leq M$. Let
$c\in\Bl(M)$ and $b\in\Bl(N)$ have the common defect group $D$, and
suppose that $c^N=b$. For every $p'$-element $x\in M$ with
$D\in\Syl_p(C_N(x))$, we have
\begin{equation}\label{r4:eq:23}
 \lambda_b\bigl(\operatorname{Cl}_N(x)^+\bigr)
 =\lambda_c\bigl(\operatorname{Cl}_M(x)^+\bigr).
\end{equation}
\end{enumerate}
\end{lemma}

\begin{proof}
For (i), Gallagher's theorem shows that $\zeta\nu$ is irreducible.
Moreover, $\zeta_L$ is irreducible because its restriction to $K$ is
irreducible.
Since $C_A(x)\leq L$,
\[
 |\operatorname{Cl}_A(x)|
 =|\operatorname{Cl}_{A/K}(\bar x)|\,|\operatorname{Cl}_L(x)|.
\]
Together with $\nu(x)=\bar\nu(\bar x)$ and
$(\zeta\nu)(1)=\zeta(1)\bar\nu(1)$, this proves
\eqref{r4:eq:22}; compare \cite[Lemma~2.2]{NS14}.

For (ii), $N_M(D)=N_N(D)$, and $b,c$ have the same Brauer correspondent
$c_0$ in this group. We claim that
\[
 \operatorname{Cl}_N(x)\cap C_N(D)
 =\operatorname{Cl}_M(x)\cap C_N(D).
\]
Indeed, if $x^n\in C_N(D)$, then $D^{n^{-1}}$ and $D$ are Sylow
$p$-subgroups of $C_N(x)$. Choose $u\in C_N(x)$ such that
$D^u=D^{n^{-1}}$. Then $un\in N_N(D)\leq M$ and $x^{un}=x^n$.
The two class sums therefore have the same $D$-Brauer image. Evaluating
their central forms through the common Brauer correspondent $c_0$ gives
\eqref{r4:eq:23}. This is the defect-class calculation in the proof of
\cite[Theorem~4.6]{NS14}. It depends only on the given block correspondence
and is independent of the multiplication result proved below.
\end{proof}

\begin{lemma}\label{r4:9}
Let $N\unlhd G$, let $H\leq G$ with $G=NH$, and set $M=N\cap H$.
Suppose that $K\unlhd G$ and $K\leq M$. Let $\xi\in\Irr(G)$ satisfy
\[
 \theta:=\xi_K\in\Irr(K),\qquad \operatorname{ht}(\theta)=0.
\]
Write $\bar G=G/K$, $\bar N=N/K$, $\bar H=H/K$, and $\bar M=M/K$.
Suppose that a fixed pair of associated projective representations
$\mathcal R,\mathcal R'$ affords a block isomorphism
\[
 (\bar G,\bar N,\bar\rho)\sim_b
 (\bar H,\bar M,\bar\rho'),\qquad
 \bar\rho\in\Irr_0(\bar N),\quad \bar\rho'\in\Irr_0(\bar M).
\]
Let $D\leq M$ be a $p$-subgroup such that $D_0=D\cap K$ is a defect
group of $\bl(\theta)$ and $\bar D=DK/K$ is a common defect group of
$\bl(\bar\rho)$ and $\bl(\bar\rho')$. Assume that
\[
 N_{\bar N}(\bar D)\leq\bar M,\qquad N_N(D)\leq M.
\]
The second containment follows from the first and is recorded for use in
the Sylow arguments below. Let $\rho,\rho'$ be the inflations of
$\bar\rho,\bar\rho'$, respectively, and set
\[
 \tau=\xi_N\rho\in\Irr(N),\qquad
 \tau'=\xi_M\rho'\in\Irr(M).
\]
Then
\begin{equation}\label{r4:eq:24}
 \tau\in\Irr_0(N\mid D)\quad\Longleftrightarrow\quad
 \tau'\in\Irr_0(M\mid D).
\end{equation}
If these conditions hold, let $\mathcal X$ be an ordinary representation
affording $\xi$, and let $\pi:G\longrightarrow G/K$ be the quotient map.
The pair
\[
 \mathcal P=\mathcal X\otimes(\mathcal R\circ\pi),\qquad
 \mathcal P'=\mathcal X_H\otimes(\mathcal R'\circ\pi_H)
\]
affords
\[
 (G,N,\tau)\sim_b(H,M,\tau').
\]
In particular, the tensor correspondences $\sigma_J$ defined by this same
pair satisfy
\begin{equation}\label{r4:eq:25}
 \bl(\sigma_J(\chi))^J=\bl(\chi)
 \qquad(N\leq J\leq G,\ \chi\in\Irr(J\mid\tau)).
\end{equation}
Neither inflation $\rho$ nor $\rho'$ is assumed to have height zero.
\end{lemma}

\begin{proof}
The restriction of $\mathcal X$ to every subgroup containing $K$ is
irreducible. Gallagher's theorem thus shows that
$\tau$ and $\tau'$ are irreducible. They are $G$-invariant and
$H$-invariant, respectively: $\xi$ is a character of $G$, and the quotient
characters are invariant under their respective overgroups.

We first show that $\xi_N$ and $\xi_M$ have height zero. If $E$ is a defect
group of $\bl(\xi_N)$, then the defect-intersection theorem shows that
$E\cap K$ is a defect group of $\bl(\theta)$. By \cite[Proposition~2.5(d)]{NS14},
$EK/K\in\Syl_p(N/K)$. Hence
\[
 |E|=|D_0|\,|N:K|_p,\qquad
 \xi_N(1)_p=\theta(1)_p=|K:D_0|_p,
\]
which proves the assertion for $\xi_N$. The argument for $\xi_M$ is
identical.

If $F$ is any defect group of $\bl(\tau)$, the intersection theorem and
\cite[Proposition~2.5(b)]{NS14} give
\[
 |F\cap K|=|D_0|,\qquad |FK/K|\geq|\bar D|.
\]
The latter inequality uses only that $\xi_N$ restricts irreducibly to $K$
and that $\bar D$ is a defect group of $\bl(\bar\rho)$. Consequently,
\begin{equation}\label{r4:eq:26}
 |F|\geq |D_0|\,|\bar D|=|D|.
\end{equation}
The same lower bound holds for every defect group of $\bl(\tau')$. Here we use only
the orders of the intersections with $K$; the intersections themselves
need not equal $D_0$.

The height-zero hypotheses on the quotient characters give
\begin{align}
 \tau(1)_p
 &=|K:D_0|_p\,|N/K:\bar D|_p=|N:D|_p,
 \label{r4:eq:27}\\
 \tau'(1)_p
 &=|K:D_0|_p\,|M/K:\bar D|_p=|M:D|_p.
 \label{r4:eq:28}
\end{align}
To obtain height zero for either product, it therefore suffices to show
that its block has defect group $D$.

Suppose first that $\tau\in\Irr_0(N\mid D)$. Choose a $p'$-element
$x\in N$ such that
\[
 D\in\Syl_p(C_N(x)),\qquad
 \lambda_\tau\bigl(\operatorname{Cl}_N(x)^+\bigr)\neq0.
\]
The characterization of block defect groups ensures the existence of such
a defect class. Since $x$ centralizes $D$, the normalizer containment
implies that $x$ lies in $M$.
Set
\[
 \bar x=xK,\qquad L/K=C_{\bar N}(\bar x),\qquad L_1=L\cap M.
\]
Then $L_1/K=C_{\bar M}(\bar x)$, $C_L(x)=C_N(x)$, and
$C_{L_1}(x)=C_M(x)$. Equation~\eqref{r4:eq:22} gives
\[
 0\neq\lambda_\tau\bigl(\operatorname{Cl}_N(x)^+\bigr)
 =\lambda_{\xi_L}\bigl(\operatorname{Cl}_L(x)^+\bigr)
  \lambda_{\bar\rho}\bigl(\operatorname{Cl}_{\bar N}(\bar x)^+\bigr).
\]
Both factors are nonzero. The Min--Max theorem gives a defect group $E$ of
$\bl(\xi_L)$ with $E\leq D$, because $D$ is a Sylow $p$-subgroup of
$C_L(x)$. Since $(\xi_L)_K=\theta$ is irreducible,
\cite[Proposition~2.5(d)]{NS14} gives $EK/K\in\Syl_p(L/K)$. As
$EK/K\leq\bar D\leq L/K$, we obtain
\[
 \bar D\in\Syl_p(L/K),\qquad
 \bar D\in\Syl_p(L_1/K),\qquad p\nmid[L:L_1].
\]
Applying \eqref{r4:eq:23} to the quotient blocks yields
\[
 \lambda_{\bar\rho'}\bigl(\operatorname{Cl}_{\bar M}(\bar x)^+\bigr)
 =\lambda_{\bar\rho}\bigl(\operatorname{Cl}_{\bar N}(\bar x)^+\bigr)
 \neq0.
\]
By Lemma~\ref{r4:7}, $D\in\Syl_p(C_M(x))$. Therefore
\[
 r=\frac{|\operatorname{Cl}_L(x)|}{|\operatorname{Cl}_{L_1}(x)|}
 =\frac{[L:L_1]}{[C_N(x):C_M(x)]}
\]
is a $p$-adic unit, although it need not be an integer. Since $\xi_L$ and
$\xi_{L_1}$ have the same degree,
\begin{equation}\label{r4:eq:29}
 \lambda_{\xi_L}\bigl(\operatorname{Cl}_L(x)^+\bigr)
 =r^*\lambda_{\xi_{L_1}}\bigl(\operatorname{Cl}_{L_1}(x)^+\bigr).
\end{equation}
Hence the local central form is nonzero. A second application of the
product formula gives $\lambda_{\tau'}(\operatorname{Cl}_M(x)^+)\neq0$.
The Min--Max theorem then gives a defect group $F'\leq D$ of
$\bl(\tau')$. By the lower bound \eqref{r4:eq:26}, we have $F'=D$,
and \eqref{r4:eq:28} gives height zero.

Conversely, assume that $\tau'\in\Irr_0(M\mid D)$. Choose a
$p'$-element $x\in M$ such that
\[
 D\in\Syl_p(C_M(x)),\qquad
 \lambda_{\tau'}\bigl(\operatorname{Cl}_M(x)^+\bigr)\neq0.
\]
Lemma~\ref{r4:7} gives $D\in\Syl_p(C_N(x))$. Again set
$L/K=C_{\bar N}(xK)$ and $L_1=L\cap M$. The product formula on the
local side gives
\[
 0\neq\lambda_{\tau'}\bigl(\operatorname{Cl}_M(x)^+\bigr)
 =\lambda_{\xi_{L_1}}\bigl(\operatorname{Cl}_{L_1}(x)^+\bigr)
  \lambda_{\bar\rho'}\bigl(\operatorname{Cl}_{\bar M}(xK)^+\bigr).
\]
Both factors are nonzero. The Min--Max theorem gives a defect group
$E_1\leq D$ of $\bl(\xi_{L_1})$. Irreducibility on $K$ and
\cite[Proposition~2.5(d)]{NS14} imply that $E_1K/K$ is Sylow in $L_1/K$.
Thus $\bar D\in\Syl_p(L_1/K)$. Applying Lemma~\ref{r4:7} in the quotient, using the quotient normalizer
containment, yields
\[
 \bar D\in\Syl_p(L/K).
\]
Hence $[L:L_1]$ is a $p'$-number, and the same class-length ratio $r$ is a
$p$-adic unit. Equation~\eqref{r4:eq:29} gives
$\lambda_{\xi_L}(\operatorname{Cl}_L(x)^+)\neq0$. The quotient
defect-class formula also gives
\[
 \lambda_{\bar\rho}\bigl(\operatorname{Cl}_{\bar N}(xK)^+\bigr)
 =\lambda_{\bar\rho'}\bigl(\operatorname{Cl}_{\bar M}(xK)^+\bigr)
 \neq0.
\]
The product formula now implies
$\lambda_\tau(\operatorname{Cl}_N(x)^+)\neq0$. Since $D$ is Sylow in
$C_N(x)$, the Min--Max theorem gives a defect group $F\leq D$ of
$\bl(\tau)$. The lower bound again forces $F=D$, and
\eqref{r4:eq:27} gives height zero. This establishes the equivalence
\eqref{r4:eq:24}.

Now assume that these equivalent conditions hold. For the specified pair
$\mathcal R,\mathcal R'$, write
$\alpha$ for the common factor set, viewed by inflation on
$G/N\cong\bar G/\bar N$. Since
$\mathcal X$ is ordinary, the tensor representations $\mathcal P,\mathcal P'$
defined in the statement both have factor set $\alpha$ and are associated
with $\tau,\tau'$.

If $c\in C_G(N)$, then $cK\in C_{\bar G}(\bar N)\leq\bar H$, so
$c\in H$. Since $\mathcal X_N$ is irreducible, $\mathcal X(c)$ is
scalar. Moreover, $\mathcal R(cK),\mathcal R'(cK)$ have the same scalar.
Consequently, $\mathcal P(c),\mathcal P'(c)$ have the same scalar as well. We have
already established the common defect group and the normalizer containment.
It remains to verify the normalized trace condition of
\cite[Theorem~4.4(i)(c)]{NS14} for this pair.

Let $h\in H$ be a $p'$-element with $D\in\Syl_p(C_N(h))$, and set
$\bar h=hK$. Suppose first that
\[
 \bar D\in\Syl_p(C_{\bar N}(\bar h)).
\]
Applying \cite[Theorem~4.4]{NS14} to the specified quotient pair, and using
the height-zero assumptions, gives
\[
 \left(\frac{|N/K|_{p'}\tr\mathcal R(\bar h)}{\bar\rho(1)_{p'}}\right)^*
 =\left(\frac{|M/K|_{p'}\tr\mathcal R'(\bar h)}{\bar\rho'(1)_{p'}}\right)^*.
\]
Multiplication by the reduction of the local integer
$|K|_{p'}\xi(h)/\theta(1)_{p'}$ gives
\begin{equation}\label{r4:eq:30}
 \left(\frac{|N|_{p'}\tr\mathcal P(h)}{\tau(1)_{p'}}\right)^*
 =\left(\frac{|M|_{p'}\tr\mathcal P'(h)}{\tau'(1)_{p'}}\right)^*.
\end{equation}
Here $\xi(1)=\theta(1)$, and we have used the $p'$-parts of the product
degrees.

Suppose next that $\bar D$ is not Sylow in $C_{\bar N}(\bar h)$.
By the quotient normalizer containment and Lemma~\ref{r4:7}, it is not
Sylow in $C_{\bar M}(\bar h)$ either. Set
\[
 A=N\langle h\rangle,\qquad B=M\langle h\rangle=A\cap H.
\]
The quotient $A/N\cong B/M$ is a cyclic $p'$-group. Choose a
one-dimensional projective representation $\mathcal Q$ of this cyclic
group with factor set equal to the corresponding restriction of $\alpha^{-1}$.
Its values may be chosen to be roots of unity: the recursive construction
along a generator requires only finitely many roots of the root-of-unity
values of $\alpha$. In particular, $\mathcal Q(hN)$ is a unit of $\mathcal O$.

Inflating $\mathcal R,\mathcal R'$ and tensoring each with the same
$\mathcal Q$ gives ordinary representations of $A,B$. Their
characters $\widetilde\rho,\widetilde\rho'$ extend $\rho,\rho'$ and
are trivial on $K$. Set
\[
 \widetilde\tau=\xi_A\widetilde\rho,\qquad
 \widetilde\tau'=\xi_B\widetilde\rho'.
\]
These characters extend $\tau,\tau'$ and are therefore irreducible. Moreover,
\[
 \widetilde\tau(h)=\mathcal Q(hN)\tr\mathcal P(h),\qquad
 \widetilde\tau'(h)=\mathcal Q(hN)\tr\mathcal P'(h).
\]
No assertion about the heights of $\widetilde\rho,\widetilde\rho'$ is
needed. Set
\[
 L_A/K=C_{A/K}(\bar h),\qquad L_B=L_A\cap B.
\]
Then $L_B/K=C_{B/K}(\bar h)$. Since $A/N$ and $B/M$ are $p'$-groups,
\[
 |L_A/K|_p=|C_{\bar N}(\bar h)|_p>|\bar D|,\qquad
 |L_B/K|_p=|C_{\bar M}(\bar h)|_p>|\bar D|.
\]
Every defect group $E$ of $\bl(\xi_{L_A})$ satisfies
$EK/K\in\Syl_p(L_A/K)$, so no such $E$ is contained in $D$. On the
other hand, $C_{L_A}(h)=C_A(h)$ and $D\in\Syl_p(C_A(h))$, since
$A/N$ is a $p'$-group. The contrapositive of the Min--Max theorem gives
\[
 \lambda_{\xi_{L_A}}\bigl(\operatorname{Cl}_{L_A}(h)^+\bigr)=0.
\]
The same argument, using $B/M$ and $D\in\Syl_p(C_M(h))$, gives
\[
 \lambda_{\xi_{L_B}}\bigl(\operatorname{Cl}_{L_B}(h)^+\bigr)=0.
\]
The product central-form identity, applied in $A$ and $B$, now yields
\[
 \lambda_{\widetilde\tau}\bigl(\operatorname{Cl}_A(h)^+\bigr)=0,
 \qquad
 \lambda_{\widetilde\tau'}\bigl(\operatorname{Cl}_B(h)^+\bigr)=0.
\]
Apply \cite[Lemma~4.2(a)]{NS14} to the height-zero characters $\tau,\tau'$
and their ordinary extensions. We obtain
\begin{align*}
 \left(\frac{|N|_{p'}\widetilde\tau(h)}{\tau(1)_{p'}}\right)^*
 &=|C_N(h)|_{p'}^*\lambda_{\widetilde\tau}
   \bigl(\operatorname{Cl}_A(h)^+\bigr)=0,\\
 \left(\frac{|M|_{p'}\widetilde\tau'(h)}{\tau'(1)_{p'}}\right)^*
 &=|C_M(h)|_{p'}^*\lambda_{\widetilde\tau'}
   \bigl(\operatorname{Cl}_B(h)^+\bigr)=0.
\end{align*}
Substitute the trace expressions for $\widetilde\tau(h)$ and
$\widetilde\tau'(h)$ and cancel the common nonzero unit
$\mathcal Q(hN)^*$. This proves \eqref{r4:eq:30} in the second case.
The same lemma ensures the required local integrality.

The two cases cover all elements required by \cite[Theorem~4.4]{NS14}.
Thus the specified pair $\mathcal P,\mathcal P'$ satisfies the conditions
on factor sets, central scalars, and normalized traces. The cited theorem
now gives the block triple isomorphism and \eqref{r4:eq:25}, with all block
correspondences afforded by the tensor pair specified in the statement.
\end{proof}

\begin{remark}\label{r4:10}
A quotient bijection on $\Irr_0(\bar N\mid\bar D)$ need not send every
quotient character to a product with defect group $D$. The relevant subset is
\[
 \left\{\bar\rho\in\Irr_0(\bar N\mid\bar D):
 \xi_N\operatorname{Inf}^{N}_{\bar N}(\bar\rho)
 \in\Irr_0(N\mid D)\right\}.
\]
By the forward implication of Lemma~\ref{r4:9}, this subset maps into the
corresponding local subset. The reverse implication gives the analogous
property for the inverse quotient bijection. The induced multiplication
correspondence is therefore a bijection. This argument does not require
inflation to preserve height zero.
\end{remark}

We now apply the multiplication lemma to the central extension and the
quotient correspondence $\Omega_Q$ constructed above. The lemma identifies the
subsets on which this correspondence yields height-zero characters with the
prescribed defect group.

For $\tau\in\Irr(\widehat X\mid\zeta_0)$, Gallagher's theorem gives a
unique $\rho\in\Irr(V)$ such that
\begin{equation}\label{r4:eq:31}
 \tau=\widetilde\zeta_{\widehat X}
       \operatorname{Inf}^{\widehat X}_{V}\rho.
\end{equation}
If $\tau\in\Irr_0(\widehat X\mid\widehat D,\zeta_0)$, then
$\rho\in\Irr_0(V\mid Q)$ by \cite[Proposition~2.5(c)]{NS14}. Indeed, the defect projection argument places a defect group of
$\bl(\rho)$ inside $Q$, while taking the ratio of the identities
\[
 |\widehat D|\tau(1)_p=|\widehat X|_p,\qquad
 |D_{00}|\zeta_0(1)_p=|K_0|_p
\]
forces equality of orders and height zero for $\rho$. Set
$\rho'=\Omega_Q(\rho)$ and define
\begin{equation}\label{r4:eq:32}
 \tau'=\widetilde\zeta_{\widehat Y}
        \operatorname{Inf}^{\widehat Y}_{N_V(Q)}\rho'.
\end{equation}
The forward implication of Lemma~\ref{r4:9} gives
$\tau'\in\Irr_0(\widehat Y\mid\widehat D,\zeta_0)$. Conversely,
apply Gallagher's theorem to any such $\tau'$. The resulting factor
$\rho'$ has height zero and defect group $Q$. With
$\rho=\Omega_Q^{-1}(\rho')$, the reverse implication of the same lemma
shows that the product in \eqref{r4:eq:31} has height zero and defect group
$\widehat D$.
Thus \eqref{r4:eq:31} and \eqref{r4:eq:32} define a bijection
\begin{equation}\label{r4:eq:33}
 \widehat\Upsilon:
 \Irr_0(\widehat X\mid\widehat D,\zeta_0)
 \longrightarrow
 \Irr_0(\widehat Y\mid\widehat D,\zeta_0).
\end{equation}
In this construction, \eqref{r4:eq:21} is restricted to the subsets whose
Gallagher products have height zero. The condition is imposed only on
these subsets, not on the entire domain of the quotient bijection.

Let $\nu:Z\longrightarrow\mathbb C^\times$ be given by $\nu(z)=z$,
and also regard it as a character of $\overline Z$, using $K_0\cap Z=1$.
Since $L(z)=\nu(z)I$, equation~\eqref{r4:eq:31} gives
\begin{equation}\label{r4:eq:34}
 Z\leq\ker\tau\quad\Longleftrightarrow\quad
 \rho_{\overline Z}=\rho(1)\nu^{-1}.
\end{equation}
By the central-scalar condition in the quotient induction relation,
\eqref{r4:eq:21} preserves this central-character fiber. Therefore
\eqref{r4:eq:33} restricts to a bijection between characters whose kernels
contain $Z$, and descends to a bijection
\begin{equation}\label{r4:eq:35}
 \Upsilon_\zeta:\Irr_0(X\mid D,\zeta)
 \longrightarrow\Irr_0(Y\mid D,\zeta).
\end{equation}
The reverse implication for height zero also applies on this central fiber,
which proves surjectivity in \eqref{r4:eq:35}.

Take $a=(h,\sigma)\in(C\times\mathcal H)_\zeta$.
Lemma~\ref{r4:2} shows that $f_a$ stabilizes $K_0$ and $\widehat X$,
and \eqref{r4:eq:15} shows that it stabilizes $\widehat D$. It therefore
fixes $Q$, so $t_a=(1,f_a)\in N_{\mathcal A}(Q)$. The mixed invariance
of the ordinary extension gives
\begin{equation}\label{r4:eq:36}
 \bigl(\widetilde\zeta\operatorname{Inf}\rho\bigr)^{(f_a,\sigma)}
 =\widetilde\zeta\operatorname{Inf}
     \bigl(\rho^{(t_a,\sigma)}\bigr).
\end{equation}
The $N_{\mathcal A}(Q)\times\mathcal H$-equivariance of
\eqref{r4:eq:21}, together with \eqref{r4:eq:36}, proves that
\eqref{r4:eq:33} and \eqref{r4:eq:35} are
$(C\times\mathcal H)_\zeta$-equivariant. The construction also shows that nonemptiness of either fiber implies
nonemptiness of the other. Together with the empty bijection chosen above,
this covers all cases.
\subsection{Descent of projective representations}\label{r4:sub:5}

We work with the central extension, quotient correspondence, and projective
representations fixed above. Recall that
\[
 A_0=A_{\zeta^{\mathcal H}},\qquad B=A_\zeta,\qquad
 E=KN_{A_0}(D),\qquad Y=KN_X(D),
\]
so that $E\cap X=Y$. Write
\[
 \pi:\widehat B\longrightarrow B,\qquad
 \ker\pi=Z\leq Z(\widehat B),\qquad
 K_0\unlhd\widehat B,\qquad K_0\cong K.
\]
We now write $\mathcal L$ for the ordinary representation $L$.
It affords $\widetilde\zeta$, its restriction to $K_0$ affords
$\zeta_0$, and
\[
 \mathcal L(z)=\nu(z)I\qquad(z\in Z),
\]
where $\nu$ is faithful. As above, let $\widehat X=\pi^{-1}(X)$ and
let $\widehat D$ be the unique Sylow $p$-subgroup of $\pi^{-1}(D)$. Set
\[
 U=K_0\widehat D,\qquad
 q:\widehat B\longrightarrow\overline B=\widehat B/K_0,\qquad
 V=\widehat X/K_0,\qquad Q=U/K_0.
\]
The preceding results give
\begin{equation}\label{r4:eq:37}
 \begin{aligned}
 \widehat Y&=N_{\widehat X}(U)
   =K_0N_{\widehat X}(\widehat D)=\pi^{-1}(Y),\\
 \widehat E&=N_{\widehat B}(U)
   =K_0N_{\widehat B}(\widehat D)=\pi^{-1}(E\cap B).
 \end{aligned}
\end{equation}
In particular, $\widehat Y/K_0=N_V(Q)$ and
$\widehat E/K_0=N_{\overline B}(Q)$. Thus the local preimage is the normalizer of $U$;
\eqref{r4:eq:37} does not identify it with the normalizer of
$\widehat D$ alone.

For $a=(h,\sigma)\in\Gamma=(A_0\times\mathcal H)_\zeta$, the
automorphisms $f_a$ constructed in Lemma~\ref{r4:2} satisfy
\begin{equation}\label{r4:eq:38}
 \begin{aligned}
 \pi(f_a(g))&=h\pi(g)h^{-1},& f_a(K_0)&=K_0,\\
 f_a(\widehat X)&=\widehat X,&
 \widetilde\zeta^{f_a\sigma}&=\widetilde\zeta.
 \end{aligned}
\end{equation}
Their image $F$ is finite. We continue to work with
$\mathcal A=\overline B\rtimes F$ and the fixed pair of projective
representations realizing the strong relation supplied by induction on
$V\unlhd\mathcal A$. Working with $\overline B$ as a subgroup preserves
its centralizer of $V$ and the corresponding scalar values.

\begin{lemma}\label{r4:11}
If $a=(h,\sigma)\in(E\times\mathcal H)_\zeta$, then $f_a(U)=U$.
Thus $t_a=(1,f_a)\in N_{\mathcal A}(Q)$, where $f_a$ is viewed as an
element of $F$.
\end{lemma}

\begin{proof}
Since $h\in KN_{A_0}(D)$, it normalizes $KD$.
Equation~\eqref{r4:eq:38} shows that $f_a$ stabilizes $\pi^{-1}(KD)$
and $K_0$. The quotient $\pi^{-1}(KD)/K_0$ is a central cyclic extension
of the $p$-group $KD/K$. Its central $p'$-part is a normal Hall
$p'$-subgroup. By the Schur--Zassenhaus theorem and centrality, the quotient has a
unique Sylow $p$-subgroup, namely $U/K_0$. Hence $f_a$ stabilizes $U$.
This argument requires neither $h\in B$ nor $h\in N_{A_0}(D)$.
\end{proof}

Fix $\chi\in\Irr_0(X\mid D,\zeta)$ and set $\tau=\chi\circ\pi$.
Let $\rho\in\Irr_0(V\mid Q)$ be the unique character satisfying
\[
 \tau=\widetilde\zeta_{\widehat X}\operatorname{Inf}(\rho).
\]
Set $\rho'=\Omega_Q(\rho)$ and
\[
 \tau'=\widetilde\zeta_{\widehat Y}\operatorname{Inf}(\rho').
\]
We henceforth omit inflation symbols when no confusion can arise.
By the preceding lemmas, $\tau'$ has height zero and $\widehat D$
is a common defect group of $\bl(\tau)$ and $\bl(\tau')$.

\begin{lemma}\label{r4:12}
We have $Z\leq\ker\tau'$. Let $\chi'\in\Irr_0(Y\mid D,\zeta)$ be
the character obtained by descending $\tau'$. Then
\begin{equation}\label{r4:eq:39}
 (E\times\mathcal H)_\chi=(E\times\mathcal H)_{\chi'}.
\end{equation}
\end{lemma}

\begin{proof}
Since $Z\leq\ker\tau$, its Gallagher factor satisfies
\[
 \rho_{q(Z)}=\rho(1)\nu^{-1},
\]
where we identify characters via $Z\cong q(Z)$. The subgroup
$q(Z)\leq Z(V)$ lies in $N_V(Q)$ and centralizes $V$. The central-scalar
condition in the quotient strong relation gives
$\rho'_{q(Z)}=\rho'(1)\nu^{-1}$, whence
$\tau'|_Z=\tau'(1)1_Z$. Descent preserves the degree, and the defect group becomes
$\widehat DZ/Z\cong D$. It therefore preserves height zero.

Let $a=(h,\sigma)\in(E\times\mathcal H)_\chi$. The restriction
$\chi_K$ is a positive integer multiple of $\zeta$, so $a\in\Gamma$.
Equation~\eqref{r4:eq:38} gives $\tau^{f_a\sigma}=\tau$. Since
$\widetilde\zeta^{f_a\sigma}=\widetilde\zeta$, the injectivity of the
Gallagher correspondence gives
\[
 \rho^{t_a\sigma}=\rho.
\]
By Lemma~\ref{r4:11}, $t_a\in N_{\mathcal A}(Q)$. The mixed equivariance
of $\Omega_Q$ gives $(\rho')^{t_a\sigma}=\rho'$. Multiplying by
$\widetilde\zeta_{\widehat Y}$ yields
$(\tau')^{f_a\sigma}=\tau'$, and hence $(\chi')^a=\chi'$.

Conversely, if $a$ fixes $\chi'$, the homogeneity of its restriction
to $K$ again gives $a\in\Gamma$. The injectivity of the Gallagher
correspondence gives
$(\rho')^{t_a\sigma}=\rho'$. Equivariance and injectivity of $\Omega_Q$
imply $\rho^{t_a\sigma}=\rho$, and hence $\chi^a=\chi$.
\end{proof}

\begin{proposition}\label{r4:13}
There is a pair of projective representations associated with
$\tau$ and $\tau'$, respectively,
\[
 \widehat{\mathcal P}:\widehat B_\tau\longrightarrow
      \GL_{\tau(1)}(\mathbb Q^{\mathrm{ab}}),
 \qquad
 \widehat{\mathcal P}':\widehat E_{\tau'}\longrightarrow
      \GL_{\tau'(1)}(\mathbb Q^{\mathrm{ab}}),
\]
with a common factor set and equal central scalars, satisfying all
intermediate-group block conditions. For every
$a=(h,\sigma)\in(E\times\mathcal H)_\chi$, the comparison functions of these representations under $f_a\sigma$ agree.
\end{proposition}

\begin{proof}
The homogeneous restrictions to $K$ give $A_\chi\leq B$ and
$E_{\chi'}\leq B$.
Inflation and \eqref{r4:eq:39} give
\[
 \widehat B_\tau=\pi^{-1}(A_\chi),\qquad
 \widehat E_{\tau'}=\pi^{-1}(E_{\chi'})
   =\widehat E\cap\widehat B_\tau.
\]
The ordinary character $\widetilde\zeta$ is invariant under conjugation
by $\widehat B$. The injectivity of the Gallagher correspondence therefore also gives
\[
 q(\widehat B_\tau)=\overline B_\rho,\qquad
 q(\widehat E_{\tau'})=N_{\overline B_\rho}(Q).
\]
The group $\overline B_\rho$ stabilizes $\bl(\rho)$, which has defect
group $Q$. The conjugacy of defect groups in $V$ implies
\[
 \overline B_\rho=VN_{\overline B_\rho}(Q).
\]
Thus the restricted triples satisfy the factorization required by the
multiplication lemma.

Let $\mathcal R,\mathcal R'$ be the fixed projective representations
supplied by induction, defined on $\mathcal A_\rho$ and
$N_{\mathcal A}(Q)_{\rho'}$, respectively. Restrict them to the
subgroups $\overline B_\rho$ and
$N_{\overline B_\rho}(Q)$, and inflate along $q$, obtaining
$\widehat{\mathcal R},\widehat{\mathcal R}'$. These restrictions preserve
the common factor set, central scalars, and all block conditions for
intermediate groups in the restricted range. Define
\begin{equation}\label{r4:eq:40}
 \widehat{\mathcal P}
   =\mathcal L|_{\widehat B_\tau}\otimes\widehat{\mathcal R},
 \qquad
 \widehat{\mathcal P}'
   =\mathcal L|_{\widehat E_{\tau'}}\otimes\widehat{\mathcal R}'.
\end{equation}
Since $\mathcal L$ is ordinary, the common factor set is the inflation
of the quotient factor set. The restrictions to $\widehat X$ and $\widehat Y$ afford $\tau$ and
$\tau'$, respectively.

If $c\in C_{\widehat B_\tau}(\widehat X)$, then $\mathcal L(c)$ is
scalar because $\mathcal L|_{K_0}$ is irreducible. The element $q(c)$
centralizes $V$, so $\widehat{\mathcal R}(c)$ and
$\widehat{\mathcal R}'(c)$ are scalar matrices with the same scalar value. Thus
$\widehat{\mathcal P}(c)$ and $\widehat{\mathcal P}'(c)$ also have
the same scalar value. Applying Lemma~\ref{r4:9} to this pair
$\widehat{\mathcal R},\widehat{\mathcal R}'$ and to $\mathcal L$
gives all intermediate-group block conditions for the pair in
\eqref{r4:eq:40}.

Now take $a=(h,\sigma)\in(E\times\mathcal H)_\chi$. We have proved
$\rho^{t_a\sigma}=\rho$ and $(\rho')^{t_a\sigma}=\rho'$.
The quotient strong relation supplies a comparison function $\kappa_a$
and invertible matrices $T_a,T_a'$ such that
\[
 \mathcal R(\overline f_a(x))^\sigma
 =\kappa_a(x)T_a\mathcal R(x)T_a^{-1}
 \qquad(x\in\overline B_\rho),
\]
and the same $\kappa_a$ occurs on the local side. The restrictions are well-defined: $t_a$
normalizes $\overline B$, and mixed stability implies that it also
normalizes $\overline B_\rho$. The same argument applies locally.

By Lemma~\ref{r4:2}, $\mathcal L^{f_a\sigma}$ and
$\mathcal L$ afford the same ordinary character of $\widehat B$.
Thus there is an invertible matrix $U_a$ such that
\[
 \mathcal L(f_a(g))^\sigma
 =U_a\mathcal L(g)U_a^{-1}\qquad(g\in\widehat B).
\]
This intertwining identity for ordinary representations involves no
scalar comparison function. For $g\in\widehat B_\tau$, we obtain
\begin{equation}\label{r4:eq:41}
 \widehat{\mathcal P}(f_a(g))^\sigma
 =\kappa_a(q(g))(U_a\otimes T_a)\widehat{\mathcal P}(g)
  (U_a\otimes T_a)^{-1}.
\end{equation}
Replacing $T_a$ by $T_a'$ gives the corresponding local identity with
the same function $\kappa_a(q(g))$. Hence the same pair of tensor product
representations satisfies both the mixed and the block conditions.
\end{proof}

\begin{lemma}\label{r4:14}
We have
\[
 C_{\widehat B_\tau}(\widehat X)/Z=C_{A_\chi}(X).
\]
\end{lemma}

\begin{proof}
The quotient homomorphism gives one inclusion. For the converse, let
$c\in C_{A_\chi}(X)$ and choose a lift $\widehat c\in\widehat B_\tau$.
Since $K_0\unlhd\widehat B$ and $K_0\cap Z=1$, centralization of $K$
by $c$ implies centralization of $K_0$ by $\widehat c$. Schur's lemma
then implies that $\mathcal L(\widehat c)$ is scalar. For $x\in\widehat X$, the
commutator $[\widehat c,x]$ lies in $Z$ and has identity image under
$\mathcal L$. The faithfulness of $\mathcal L|_Z$ gives
$[\widehat c,x]=1$. The restriction of $\pi$ to this centralizer has kernel $Z$, which
proves the assertion.
\end{proof}

\begin{theorem}\label{r4:15}
For the pair $\chi\leftrightarrow\chi'$ constructed above,
\[
 (A_{\chi^{\mathcal H}},X,\chi)_{\mathcal H}
 \geq_b(E_{(\chi')^{\mathcal H}},Y,\chi')_{\mathcal H}.
\]
The mixed comparison functions and all intermediate-group block
correspondences are afforded by a single pair of projective representations.
\end{theorem}

\begin{proof}
We first descend the pair of projective representations. Both $\tau$ and $\tau'$ have $Z$ in their
kernels, and $Z\leq\widehat X\cap\widehat Y$. Since the associated
factor sets are inflated from the quotients by $\widehat X$ and
$\widehat Y$, respectively,
\[
 \widehat{\mathcal P}(zg)=\widehat{\mathcal P}(g),\qquad
 \widehat{\mathcal P}'(zg)=\widehat{\mathcal P}'(g)
 \qquad(z\in Z).
\]
We may therefore define
\[
 \mathcal P(\pi(g))=\widehat{\mathcal P}(g),\qquad
 \mathcal P'(\pi(g))=\widehat{\mathcal P}'(g).
\]
These representations are associated with $\chi$ and $\chi'$,
respectively, and have a common factor set. Lemma~\ref{r4:14} shows that every element of the quotient
centralizer lifts to the centralizer before passing to the quotient. Hence equality of the
central scalars is preserved.

The proof of \cite[Corollary~4.5]{NS14} gives the intermediate-group
block conditions for this pair of quotient representations. We include
the trace argument to verify these conditions for the specified pair. Let $x$ be a local $p'$-test
element in the quotient. Choose a $p'$-lift $\widehat x$ in its central
preimage. Since $\widehat DZ/Z=D$, centrality of $Z$ and coprime action
imply that $\widehat x$ centralizes $\widehat D$. The quotient map gives
\[
 |C_{\widehat X}(\widehat x)|_p
 \leq |Z|_p\,|C_X(x)|_p
 =|Z|_p\,|D|=|\widehat D|.
\]
Together with $\widehat D\leq C_{\widehat X}(\widehat x)$, this yields
$\widehat D\in\Syl_p(C_{\widehat X}(\widehat x))$. Apply the
normalized trace identity for $\widehat{\mathcal P},\widehat{\mathcal P}'$
at $\widehat x$. Since
$\widehat{\mathcal P}(\widehat x)=\mathcal P(x)$ and
$\widehat{\mathcal P}'(\widehat x)=\mathcal P'(x)$, the common unit
$|Z|_{p'}$ cancels. The result is precisely the normalized trace identity
for $\mathcal P,\mathcal P'$. Applying \cite[Theorem~4.4]{NS14} to
these quotient representations gives
\begin{equation}\label{r4:eq:42}
 \bl(\tau_W(\omega))^W=\bl(\omega)
 \qquad\bigl(X\leq W\leq A_\chi,\ \omega\in\Irr(W\mid\chi)\bigr).
\end{equation}
The tensor correspondences also descend: any projective
representation used to cancel the factor set is inflated from $W/X$,
and hence is trivial on the central kernel. The common defect group
descends to $D$, and $N_X(D)\leq Y$.

We next descend the comparison functions for the same pair. The automorphism
$f_a$ stabilizes $Z$ and induces conjugation by $h$ on the quotient
under $\pi$. Thus \eqref{r4:eq:41} gives
\[
 \mathcal P(hxh^{-1})^\sigma
 =\mu_a(x)(U_a\otimes T_a)\mathcal P(x)(U_a\otimes T_a)^{-1},
\]
with the same $\mu_a$ on the local side. The formula
$\mu_a(\pi(g))=\kappa_a(q(g))$ is well-defined because $\kappa_a$ is
constant on $V$-cosets and $q(Z)\leq V$. Consequently, the quotient
comparison functions satisfy $\mu_a'=\mu_a|_{E_{\chi'}}$.

It remains to identify the stabilizers of the Galois orbits. Set
\[
 S=A_{\chi^{\mathcal H}},\qquad T=E_{(\chi')^{\mathcal H}}.
\]
Since $\chi_K$ is a multiple of $\zeta$, we have $S\leq A_0$.
Equation~\eqref{r4:eq:39} shows that, for $h\in E$, the condition
$\chi^h\in\chi^{\mathcal H}$ is equivalent to
$(\chi')^h\in(\chi')^{\mathcal H}$. Hence $T=E\cap S$. If $g\in S$,
then $\bl(\chi)^g$ is a Galois conjugate of $\bl(\chi)$. Both $D^g$
and $D$ are defect groups of that block, so they are conjugate in $X$.
Thus $S=XN_S(D)$. Since $N_S(D)\leq N_{A_0}(D)\leq E$,
\[
 S=XT,\qquad X\cap T=X\cap E=Y.
\]
Furthermore,
\[
 C_S(X)\leq C_S(D)\leq N_S(D)\leq T.
\]
Restricting \eqref{r4:eq:39} to $T\times\mathcal H$ gives the required
equality of mixed stabilizers. Finally, $S_\chi=A_\chi$ and
$T_{\chi'}=E_{\chi'}$, so the representations constructed above are
defined on precisely the required inertia groups. Their common factor set,
equal central scalars, equal mixed comparison functions, common defect
group $D$, normalizer containment, and \eqref{r4:eq:42} prove the full
block $\mathcal H$-triple relation.
\end{proof}

\subsection{The semilinear descent theorem}\label{r4:sub:6}

We now assemble the correspondences on individual fibers over the
Galois orbit to obtain the descent theorem.

\begin{theorem}[Semilinear centralization]\label{r4:16}
Assume Hypothesis~\ref{r4:1}. Let
\[
 K\unlhd A,\qquad K\leq X\unlhd A,\qquad D\leq X,
\]
where $D$ is a $p$-subgroup. Set $D_0=K\cap D$, and suppose that
\[
 \zeta\in\Irr_0(K\mid D_0),\qquad
 \zeta\text{ is }X\text{-invariant},\qquad |X:K|<n.
\]
Set
\[
 A_0=A_{\zeta^{\mathcal H}},\qquad C=N_{A_0}(D),\qquad
 Y=KN_X(D),\qquad E=KC.
\]
There is a $C\times\mathcal H$-equivariant bijection
\begin{equation}\label{r4:eq:43}
 \Upsilon:\Irr_0(X\mid D,\zeta^{\mathcal H})
 \longrightarrow\Irr_0(Y\mid D,\zeta^{\mathcal H})
\end{equation}
with the following properties.
\begin{enumerate}[label=\textup{(\roman*)}]
\item For every $\gamma\in\mathcal H$, the $\zeta^\gamma$-fiber is
mapped bijectively onto the corresponding $\zeta^\gamma$-fiber.
\item If $\chi'=\Upsilon(\chi)$, then
\[
 \bl(\chi')^X=\bl(\chi),\qquad
 \frac{\chi(1)_p}{\chi'(1)_p}=|X:Y|_p,
\]
and $D$ is a common defect group of the two blocks.
\item For every corresponding pair,
\begin{equation}\label{r4:eq:44}
 (A_{\chi^{\mathcal H}},X,\chi)_{\mathcal H}
 \geq_b(E_{(\chi')^{\mathcal H}},Y,\chi')_{\mathcal H}.
\end{equation}
\end{enumerate}
\end{theorem}

We have $X\leq A_0$, $Y\unlhd E$, and $X\cap E=Y$.
Neither $Y\unlhd X$ nor $A_0=XE$ is assumed. The required factorization
holds in the overgroups stabilizing the relevant Galois orbits. The degree formula concerns
the relative $p$-parts; it does not assert an equality of the full degrees.

\begin{proof}[Proof of Theorem~\ref{r4:16}]
For $\chi\in\Irr_0(X\mid D,\zeta)$ and $\sigma\in\mathcal H$, define
\begin{equation}\label{r4:eq:45}
 \Upsilon(\chi^\sigma)=\Upsilon_\zeta(\chi)^\sigma.
\end{equation}
To see that this is well-defined, suppose that
$\chi_1^{\sigma_1}=\chi_2^{\sigma_2}$. Uniqueness of the constituent
of the restriction to $K$ gives
$\zeta^{\sigma_1}=\zeta^{\sigma_2}$. Thus
$\delta=\sigma_1\sigma_2^{-1}\in\mathcal H_\zeta$ and
$\chi_2=\chi_1^\delta$. Equivariance of \eqref{r4:eq:35} under
$(1,\delta)$ shows that the two values in \eqref{r4:eq:45} agree.
The bijections on individual fibers therefore combine to give a
bijection on their union, preserving every $\zeta^\sigma$-fiber.

For $c\in C$, choose $\gamma\in\mathcal H$ with $\zeta^c=\zeta^\gamma$.
Then $(c,\gamma^{-1})\in(C\times\mathcal H)_\zeta$, and equivariance
of \eqref{r4:eq:35} gives
\[
 \Upsilon_\zeta(\chi^{c\gamma^{-1}})
 =\Upsilon_\zeta(\chi)^{c\gamma^{-1}}.
\]
Together with \eqref{r4:eq:45}, this gives
$\Upsilon(\chi^c)=\Upsilon(\chi)^c$. Since the group and Galois
actions commute, $\Upsilon$ is $C\times\mathcal H$-equivariant.

The pair constructed above satisfies all the required properties on the
$\zeta$-fiber. For the other fibers, apply $\sigma$ simultaneously to
both projective representations. Equalities of factor sets, central scalars, and
comparison functions are preserved, as are block induction, defect
groups, and heights. The full relation therefore holds on every
$\zeta^\sigma$-fiber. The bijection must be independent of the choice of fiber representative;
the projective representations realizing the relation for a given pair
of characters need not be unique.

Taking $W=X$ in the strong relation gives
\[
 \bl(\Upsilon(\chi))^X=\bl(\chi).
\]
Since both characters have height zero and their blocks have defect
group $D$,
\[
 \chi(1)_p=\frac{|X|_p}{|D|},\qquad
 \Upsilon(\chi)(1)_p=\frac{|Y|_p}{|D|}.
\]
Taking the ratio proves the degree formula and completes the proof.
\end{proof}

\begin{remark}\label{r4:17}
The strict inequality in \eqref{r4:eq:19} ensures that the quotient lies
below the external induction bound $n$, while \eqref{r4:eq:21} applies
Hypothesis~\ref{r4:1}. The theorem therefore constitutes a complete
step in the central-index induction. Its application in the full
reduction requires both the strong relation and the quantification over
all finite overgroups in that hypothesis.
\end{remark}

The following formulation expresses the induction bound in the notation
used in the reduction. The bound is the external integer $n$, rather
than the central index of the intermediate group $T$.

\begin{proposition}\label{r4:18}
Assume Hypothesis~\ref{r4:1} for a positive integer $n$.
Let $K\leq T$ be normal subgroups of $B$,
let $D\leq T$ be a $p$-subgroup, and assume $|T:K|<n$. If
$\zeta\in\Irr_0(K\mid K\cap D)$ is $T$-invariant, there is an
$N_{B_{\zeta^{\mathcal H}}}(D)\times\mathcal H$-equivariant strong
block bijection
\[
 \Irr_0(T\mid D,\zeta^{\mathcal H})
 \longrightarrow\Irr_0(KN_T(D)\mid D,\zeta^{\mathcal H}).
\]
It preserves the fibers over the characters of $K$. For a corresponding
pair $\gamma,\gamma'$, the local overgroup in the strong relation is
$\bigl(KN_{B_{\zeta^{\mathcal H}}}(D)\bigr)_{(\gamma')^{\mathcal H}}$.
\end{proposition}

\begin{proof}
Apply Theorem~\ref{r4:16} with $X=T$. The associated projective
representation gives a finite central extension. Write $K_0$ for the
embedded copy of $K$, $\widehat T$ for the inverse image of $T$, and
$Z_1$ for the central scalar subgroup. Induction is applied only to the base group $W=\widehat T/K_0$, for
which
\[
 m(W)\leq |W:Z_1K_0/K_0|=|T:K|<n.
\]
The action constructed in Lemma~\ref{r4:2} is realized in a finite
semidirect product containing the original quotient overgroup. This is
therefore an overgroup covered by the induction hypothesis.

The remaining steps consist of restricting this correspondence to
central character fibers, taking tensor products of representations,
and descending through the kernel. They do not require the additional inequality $|T:K|<m(T)$.
The calculation of comparison functions above gives the strong relation,
and the height-zero argument in both directions gives surjectivity. This
proves the proposition.
\end{proof}

\section{Quasisimple components}\label{sec:components}

In this section we derive the consequences of Definition~\ref{r6:1}
for normal subgroups whose central quotients are direct products of
non-abelian simple groups. Starting with universal covering groups, we
form tensor products, pass to central quotients, and then transfer the
resulting relations to arbitrary ambient groups. At each stage, we verify
the block conditions and the mixed comparison functions for the same
pair of projective representations.

\subsection{Universal covers and quasisimple quotients}\label{r6:sub:3}

Passing from a universal covering group to a quasisimple quotient
requires us to transport the projective representations and track
the defect groups under the quotient map. The kernel of a central
character is normal in the relevant ambient group, but need not be
central in that group.

\begin{lemma}[Inflation through a normal kernel]\label{r6:3}
Let $K\trianglelefteq A$, let $H\leq A$, and set $M=K\cap H$.
Suppose that $T\trianglelefteq A$ and $T\leq M$, and that
\[
 (A/T,K/T,\bar\theta)_{\mathcal H}
 \geq_b(H/T,M/T,\bar\varphi)_{\mathcal H}.
\]
If the inflations $\theta$ and $\varphi$ have height zero and their
blocks have a common defect group $D$, then
\[
 (A,K,\theta)_{\mathcal H}\geq_b(H,M,\varphi)_{\mathcal H}.
\]
No assumption that $T\leq Z(A)$ is required.
\end{lemma}

\begin{proof}
Let $q$ denote the quotient map, and choose
$\bar{\mathcal P},\bar{\mathcal P}'$ affording the relation on the
quotient. Set
$\mathcal P=\bar{\mathcal P}\circ q$ and
$\mathcal P'=\bar{\mathcal P}'\circ q$. The factor sets, central
scalars, and mixed comparison functions inflate simultaneously; in
particular, $\mu_a=\bar\mu_{\bar a}\circ q$.

Every character lying over $\theta$ in an intermediate subgroup has
$T$ in its kernel. The tensor correspondences afforded by
$\mathcal P,\mathcal P'$ are therefore obtained by inflating those afforded by
$\bar{\mathcal P},\bar{\mathcal P}'$. The block induction equalities
follow from \cite[Proposition~2.4(a) and Lemma~3.12]{NS14}. Since
$\theta$ and $\varphi$ have height zero, $DT/T$ is a common defect
group of the blocks of $K/T$ and $M/T$ containing $\bar\theta$ and
$\bar\varphi$, respectively. Thus the defect-group hypothesis of
that lemma is satisfied. If
$n\in N_K(D)$, then
\[
 nT\in N_{K/T}(DT/T)\leq M/T,
\]
so $n\in M$. Hence $N_K(D)\leq M$, completing the verification of
the defining conditions after inflation.
\end{proof}

The next lemma supplies the centralizer calculation used in the
central-section trace test when passing from the quotient in
Definition~\ref{r6:1} to products of components.

\begin{lemma}\label{r6:4}
Let $V$ be perfect, let $\eta\in\Irr(V)$ have faithful central
character, and let $R$ be a defect group of $\bl(\eta)$. Let $a$ be
an automorphism of $V$ of finite $p'$-order stabilizing $\eta$. If
\[
 RZ(V)/Z(V)\in\Syl_p(C_{V/Z(V)}(a)),
\]
then $a$ centralizes $R$ and $R\in\Syl_p(C_V(a))$.
\end{lemma}

\begin{proof}
Since $\eta^a=\eta$ and the central character of $\eta$ is faithful, $a$ fixes
$Z(V)$ pointwise. Moreover, $Z(V)_p\leq R$, and $R$ is the unique
Sylow $p$-subgroup of $RZ(V)$. Hence $a$ stabilizes $R$. The action of $a$
is trivial on both $R/Z(V)_p$ and $Z(V)_p$, so coprime action gives
$[R,a]=1$.

To compare the orders of the centralizers, let $Q$ be a Sylow $p$-subgroup of
$C_{V/Z(V)}(a)$. Its preimage has a unique Sylow $p$-subgroup $Q_1$.
This subgroup is $a$-invariant and, by the same coprime action
argument, is centralized by $a$. Thus the image of
$C_V(a)\longrightarrow C_{V/Z(V)}(a)$ contains a Sylow $p$-subgroup.
Consequently,
\[
 |C_V(a)|_p=|Z(V)|_p\,|C_{V/Z(V)}(a)|_p.
\]
Together with $Z(V)_p\leq R$, this proves the assertion. The argument
uses only equality of the $p$-parts of these orders and does not require
surjectivity of the centralizer map.
\end{proof}

Lemma~\ref{r6:4} yields the central-section trace criterion of
\cite[Definition~7.2(iii)]{Sp13} for the pair in
Definition~\ref{r6:1}. More precisely, suppose that a $p'$-element $x$
stabilizing $\vartheta$ induces an automorphism of $U/Z(U)$ for which
$RZ(U)/Z(U)$ is a Sylow $p$-subgroup of the centralizer. Its action
on $V=U/W$ then satisfies the hypothesis of
\cite[Theorem~4.4]{NS14}. For the representations
$\mathcal P,\mathcal P'$ inflated from the faithful-central-character
quotient, this gives
\begin{equation}\label{r6:eq:3}
 \left(\frac{|U|_{p'}\tr\mathcal P(x)}{\vartheta(1)_{p'}}\right)^*
 =
 \left(\frac{|L_R|_{p'}\tr\mathcal P'(x)}{\varphi(1)_{p'}}\right)^*.
\end{equation}
Replacing $V$ by $U$ multiplies both sides by $|W|_{p'}^*$.
The central-section construction multiplies both representation matrices
by the same root of unity and therefore preserves the equality. Thus
\eqref{r6:eq:3} follows from the stated inductive condition and
requires no additional hypothesis on simple groups.

\subsection{Products and permutations of components}\label{r6:sub:4}

We apply the tensor and permutation constructions for $\mathcal H$-triples
to the local data of Definition~\ref{r6:1}. We use explicit projective representations
to verify the block condition after quotienting by a possibly diagonal
central subgroup.

Let $U_i$ be universal covering groups of non-abelian simple groups,
and set
\[
 \widetilde K=\prod_{i=1}^r U_i.
\]
Let $\widetilde D=\prod_iD_i$ be a block defect group. Choose the
local data compatibly within each isomorphism class and each automorphism
orbit of radical subgroups. Set $M_i=L_{D_i}$ and
\[
 \widetilde M=\prod_iM_i.
\]
We also use $\otimes$ to denote the external tensor product of characters;
thus
\[
 (\vartheta_1\otimes\cdots\otimes\vartheta_r)(u_1,\ldots,u_r)
   =\prod_{i=1}^r\vartheta_i(u_i)
 \qquad (u_i\in U_i).
\]
Define
\[
 \widetilde\Omega:\Irr_0(\widetilde K\mid\widetilde D)
   \longrightarrow\Irr_0(\widetilde M\mid\widetilde D)
\]
by
\[
 \widetilde\Omega(\vartheta_1\otimes\cdots\otimes\vartheta_r)
 =\omega_{D_1}(\vartheta_1)\otimes\cdots
   \otimes\omega_{D_r}(\vartheta_r).
\]
Blocks and defect groups of a direct product decompose as products
of blocks and defect groups of its factors, and character heights are additive. Thus $\widetilde\Omega$ is a bijection on
the indicated height-zero sets. It preserves central characters and
blocks with a common Brauer correspondent, and is
$\widetilde\Gamma\times\mathcal H$-equivariant, where
$\widetilde\Gamma=\Aut(\widetilde K)_{\widetilde D}$.

Fix $\Theta=\bigotimes_i\vartheta_i$ and set
$\Phi=\widetilde\Omega(\Theta)$. Inflate the representations in
Definition~\ref{r6:1} through the individual central kernels to the
groups $U_i$, and denote the resulting pairs by
$\mathcal P_i,\mathcal P_i'$. After making the chosen transports, define the tensor permutation
representation on a wreath product of equal character factors by
\begin{equation}\label{r6:eq:4}
 \mathcal P((x_1,\ldots,x_s)\pi)
   =\bigl(\mathcal P_1(x_1)\otimes\cdots\otimes\mathcal P_s(x_s)\bigr)
      T_\pi.
\end{equation}
On the local side, use $\mathcal P_i'$ together with the corresponding
tensor permutation matrices. The entries of $T_\pi$ lie in $\{0,1\}$. The projective representations
for characters in the same orbit are obtained from those for a fixed
orbit representative by automorphism and Galois conjugation.

The constructions for $\geq_c$ are those of
\cite[Lemmas~2.5--2.6 and Theorem~2.7]{NSV20}. We record the comparison
functions for the representations chosen here. For a diagonal mixed
element $a=(h,\sigma)$ acting on a block of equal factors, they satisfy
\begin{equation}\label{r6:eq:5}
 \begin{aligned}
 \mu_a((x_1,\ldots,x_s)\pi)&=\prod_{i=1}^s\mu_{h\sigma}(x_i),\\
 \mu_a'((x_1,\ldots,x_s)\pi)&=\prod_{i=1}^s\mu_{h\sigma}'(x_i).
 \end{aligned}
\end{equation}
These identities follow by tensoring the intertwining matrices for the
individual factors and using $T_\pi^\sigma=T_\pi$. For a general mixed stabilizer, partition
the factors according to their $\Aut(U_i)_{D_i}\times\mathcal H$-orbits
and multiplicities. On each resulting block, apply the preceding construction to the
stabilizers of the individual direct factors, permutations of these
factors, and diagonal mixed elements.
The composition rule for comparison functions extends
\eqref{r6:eq:5} to all mixed stabilizers. The resulting pair
$\mathcal P,\mathcal P'$ therefore has common factor sets, common
central scalars, and equal mixed comparison functions.

\begin{lemma}\label{r6:5}
Set
\[
 W_\Theta=Z(\widetilde K)\cap\ker\Theta,\qquad
 F=\widetilde K/W_\Theta,\qquad L=\widetilde M/W_\Theta.
\]
The pair of tensor permutation representations constructed above
descends to a pair affording
\begin{equation}\label{r6:eq:6}
 \bigl(F\rtimes\widetilde\Gamma_{\Theta^{\mathcal H}},F,
       \bar\Theta\bigr)_{\mathcal H}
 \geq_b
 \bigl(L\rtimes\widetilde\Gamma_{\Theta^{\mathcal H}},L,
       \bar\Phi\bigr)_{\mathcal H}.
\end{equation}
\end{lemma}

\begin{proof}
First, $Z(F)=Z(\widetilde K)/W_\Theta$. Indeed, if the image of
$g\in\widetilde K$ is central in $F$, then
$k\mapsto[k,g]\in W_\Theta$ is a homomorphism from a perfect group
to an abelian group. It is trivial, so $g\in Z(\widetilde K)$.

The subgroup $W_\Theta$ is invariant under
$\widetilde\Gamma_{\Theta^{\mathcal H}}$ and is contained in
$\ker\Phi$. The associated representations are the identity on
$W_\Theta$, so they are constant on its cosets and descend to the
quotient. The centralizer condition also descends. If an automorphism
$a$ of $\widetilde K$ acts trivially modulo $W_\Theta$, then
$k\mapsto[k,a]$ is a homomorphism $\widetilde K\to W_\Theta$ and is
therefore trivial. Thus, for
$G=\widetilde K\rtimes\widetilde\Gamma_{\Theta^{\mathcal H}}$,
\[
 C_{G/W_\Theta}(\widetilde K/W_\Theta)
   =C_G(\widetilde K)/W_\Theta.
\]
It follows, also by \cite[Lemma~2.4]{NSV20}, that the descended pair
satisfies all conditions for $\geq_c$. It remains to verify the block condition for this same pair.

Let $\bar D=\widetilde D W_\Theta/W_\Theta$, a common defect group
of the blocks of $\bar\Theta$ and $\bar\Phi$. Let $x$ be a
$p'$-element of the local ordinary inertia group such that
$\bar D\in\Syl_p(C_F(x))$. Since the central character of
$\bar\Theta$ is faithful, $x$ fixes $Z(F)$ pointwise. The centralizer
argument in the proof of Lemma~\ref{r6:4}, which shows that the image
contains a Sylow $p$-subgroup, gives
\begin{equation}\label{r6:eq:7}
 \widetilde D Z(\widetilde K)/Z(\widetilde K)
 \in\Syl_p\bigl(C_{\widetilde K/Z(\widetilde K)}(x)\bigr).
\end{equation}
In particular, $x$ fixes this image of the defect group pointwise.

Suppose that $D_i\not\leq Z(U_i)$. Its image in the simple direct
factor $U_i/Z(U_i)$ of $\widetilde K/Z(\widetilde K)$ is non-trivial.
The permutation of these direct factors induced by $x$ must fix $i$: otherwise a non-identity element
of this image would be sent to a different direct factor. Hence every
non-trivial permutation cycle involves only factors $U_i$ with
$D_i\leq Z(U_i)$. For these factors, $M_i=U_i$ and
$\mathcal P_i'=\mathcal P_i$ by convention, so their contributions
to the two traces agree along each such cycle.

For each remaining factor $U_i$ stabilized by $x$, the induced
action $x_i$ has $p'$-order and stabilizes $\vartheta_i$.
The direct product decomposition and
\eqref{r6:eq:7} give
\[
 D_iZ(U_i)/Z(U_i)\in\Syl_p(C_{U_i/Z(U_i)}(x_i)).
\]
We may therefore apply \eqref{r6:eq:3} to these factors. To
evaluate the matrices, choose a $p'$-preimage of $x$ in
$\widetilde K\rtimes\widetilde\Gamma_\Theta$. Its component in the
semidirect product associated with each stabilized factor is again a
$p'$-element.

The trace in \eqref{r6:eq:4} factors over the permutation cycles.
Multiplying the identities \eqref{r6:eq:3} for the factors stabilized
by $x$ and using the
identical matrices on each non-trivial cycle yields
\[
 \left(\frac{|\widetilde K|_{p'}\tr\mathcal P(x)}
                  {\Theta(1)_{p'}}\right)^*
 =
 \left(\frac{|\widetilde M|_{p'}\tr\mathcal P'(x)}
                  {\Phi(1)_{p'}}\right)^*.
\]
Any common scalar adjustment made when assembling the orbit
representatives multiplies both sides by the same scalar, so the
equality is preserved.
Here $\Theta(1)_{p'}=\prod_i\vartheta_i(1)_{p'}$ and
$\Phi(1)_{p'}=\prod_i\varphi_i(1)_{p'}$, where
$\varphi_i=\omega_{D_i}(\vartheta_i)$; the normalizing factors also
agree for the factors $U_i$ with $D_i\leq Z(U_i)$. Cancelling the common factor
$|W_\Theta|_{p'}$ gives the trace equality for $F$ and $L$.

By \cite[Theorem~4.4]{NS14}, this pair satisfies the block condition
on every intermediate subgroup. Moreover, $N_F(\bar D)\leq L$,
since $\widetilde D$ is the unique Sylow $p$-subgroup of the central
preimage $\widetilde D W_\Theta$. This proves \eqref{r6:eq:6}.
\end{proof}

The subgroup $W_\Theta$ may contain diagonal central subgroups on
which permutations of the direct factors act non-trivially. The trace argument
above also applies in this case; it does not require the kernel to be
central in the ambient group.

\subsection{Central products and arbitrary ambient groups}\label{r6:sub:5}

We first establish the component transfer result for a perfect central
product by transferring the representations from the preceding
subsection to an arbitrary ambient group.

\begin{proposition}\label{r6:6}
Let $K\trianglelefteq A$ be perfect, and suppose that $K/Z(K)$ is a
direct product of non-abelian simple groups satisfying
Definition~\ref{r6:1}. Let $D_0$ be a non-central radical $p$-subgroup
of $K$. There exist an $N_A(D_0)$-invariant subgroup
$N_K(D_0)\leq M<K$ and an $N_A(D_0)\times\mathcal H$-equivariant
bijection
\[
 \Omega:\Irr_0(K\mid D_0)\longrightarrow\Irr_0(M\mid D_0)
\]
preserving central characters and blocks with a common Brauer
correspondent. For $H=MN_A(D_0)$, $\theta\in\Irr_0(K\mid D_0)$,
and $\varphi=\Omega(\theta)$, the same construction gives
\[
 (A_{\theta^{\mathcal H}},K,\theta)_{\mathcal H}
 \geq_b(H_{\varphi^{\mathcal H}},M,\varphi)_{\mathcal H}.
\]
\end{proposition}

\begin{proof}
If $\Irr_0(K\mid D_0)$ is empty, take $M=N_K(D_0)$ and the empty
bijection. Brauer's first main theorem shows that the local set is also
empty. Moreover, $N_K(D_0)<K$, since $O_p(K)\leq Z(K)$. We may
therefore assume that $\Irr_0(K\mid D_0)$ is non-empty. In particular,
$D_0$ is a block defect group.

There is a central epimorphism
\[
 \pi:\widetilde K=\prod_iU_i\longrightarrow K,\qquad
 C=\ker\pi\leq Z(\widetilde K),
\]
where the $U_i$ are universal covering groups. Let $\widetilde D$
be the unique Sylow $p$-subgroup of $\pi^{-1}(D_0)$. It is a defect
group of the corresponding covering block, and therefore has the form
$\prod_iD_i$. We have
\begin{equation}\label{r6:eq:8}
 \pi^{-1}(D_0)=\widetilde D C_{p'},\qquad
 \pi\bigl(N_{\widetilde K}(\widetilde D)\bigr)=N_K(D_0).
\end{equation}
Set $M=\pi(\widetilde M)$. Since
$C\leq Z(\widetilde K)\leq\widetilde M$ and
$\widetilde\Omega$ preserves central characters, it preserves the
condition of having $C$ in the kernel, and hence descends to a
bijection $\Omega$. The block correspondences through central $p$-kernels
and central $p'$-kernels preserve height zero and identify the defect
groups with their asserted images. Thus the descended map has the stated domain
and codomain and preserves blocks with a common Brauer correspondent.

The kernel $C$ may be diagonal. Indeed, if
$\Theta=\bigotimes_i\vartheta_i$ and $\lambda_i$ is the central
character of $\vartheta_i$, then
\[
 C\leq\ker\Theta
 \quad\Longleftrightarrow\quad
 \prod_i\lambda_i(c_i)=1\quad\text{for every }(c_i)_i\in C.
\]
Preservation of each $\lambda_i$ therefore preserves this condition,
even when the individual central characters are non-trivial.

Every automorphism of $K$ lifts uniquely to an automorphism of
$\widetilde K$ stabilizing $C$. Uniqueness follows from perfectness:
the difference of two lifts is a central automorphism and determines
a homomorphism into an abelian central subgroup. By \eqref{r6:eq:8},
the lifted action of $N_A(D_0)$ stabilizes $\widetilde D$. Hence $M$
is invariant under $N_A(D_0)$ and $\Omega$ is jointly equivariant. Since $D_0$ is
non-central, some $D_i$ is non-central, so $M_i<U_i$. As
$C\leq\widetilde M$, it follows that $M<K$.

Fix $\theta$, and set
\[
 \Theta=\theta\circ\pi,\qquad
 W=Z(\widetilde K)\cap\ker\Theta,\qquad
 T=W/C=Z(K)\cap\ker\theta.
\]
Then $C\leq W$ and $\widetilde K/W\cong K/T$.
Lemma~\ref{r6:5} gives the relation \eqref{r6:eq:6} in the corresponding
semidirect products, with $F$ identified with $K/T$ and $\bar\Theta$
having faithful central character. Restrict the associated projective
representations to the inverse image of the action of
$A_{\theta^{\mathcal H}}/T$ on $K/T$, and
apply \cite[Lemma~1.10]{ChenCT}. This gives the relation on the actual
quotient overgroup. Here $T\trianglelefteq A_{\theta^{\mathcal H}}$,
and perfectness gives
\begin{equation}\label{r6:eq:9}
 C_{A_{\theta^{\mathcal H}}/T}(K/T)
   =C_{A_{\theta^{\mathcal H}}}(K)/T.
\end{equation}
Indeed, if $a$ acts trivially on $K/T$, then
$k\mapsto[k,a]\in T$ is a homomorphism and must be trivial.

We next identify the local subgroup. Since $T\leq Z(K)$ and $T_p\leq D_0$,
the unique Sylow $p$-subgroup of $D_0T=D_0T_{p'}$ is $D_0$.
Therefore
\[
 N_{A_{\theta^{\mathcal H}}/T}(D_0T/T)
   =N_{A_{\theta^{\mathcal H}}}(D_0)T/T.
\]
Conjugacy of defect groups and preservation of defect groups under
$\mathcal H$ also give
\[
 A_{\theta^{\mathcal H}}=K N_{A_{\theta^{\mathcal H}}}(D_0).
\]
In the induced action on $K/T$, the image of the local subgroup is generated by
$\Inn(M/T)$ and the relevant automorphisms normalizing $D_0T/T$.
An induced automorphism belongs to this image precisely when composing
it with an inner automorphism from $M/T$ makes it normalize $D_0T/T$. Its preimage in
the actual overgroup is consequently
$(M N_{A_{\theta^{\mathcal H}}}(D_0))/T$. By mixed equivariance of
$\Omega$, this is $H_{\varphi^{\mathcal H}}/T$. Finally, apply
Lemma~\ref{r6:3} through the normal kernel $T$. This proves the
asserted relation, since the inflation argument requires only normality
of $T$ in the overgroup.
\end{proof}

\subsection{The component transfer theorem}\label{r6:sub:6}

Adjoining the central factor gives the component transfer theorem in
the form required for the reduction, with a proper local subgroup.

\begin{theorem}\label{r6:7}
Let $K\trianglelefteq A$, and suppose that $K/Z(K)$ is a direct
product of non-abelian simple groups satisfying Definition~\ref{r6:1}.
Let $D_0$ be a non-central radical $p$-subgroup of $K$. There exist
an $N_A(D_0)$-invariant subgroup and an
$N_A(D_0)\times\mathcal H$-equivariant bijection
\[
 N_K(D_0)\leq M<K,\qquad
 \Omega:\Irr_0(K\mid D_0)\longrightarrow\Irr_0(M\mid D_0)
\]
preserving central characters and blocks with a common Brauer
correspondent. Set $H=MN_A(D_0)$. For every
$\theta\in\Irr_0(K\mid D_0)$ and $\varphi=\Omega(\theta)$, we have
\begin{equation}\label{r6:eq:10}
 (A_{\theta^{\mathcal H}},K,\theta)_{\mathcal H}
 \geq_b(H_{\varphi^{\mathcal H}},M,\varphi)_{\mathcal H}.
\end{equation}
\end{theorem}

\begin{proof}
If $\Irr_0(K\mid D_0)$ is empty, take $M=N_K(D_0)$ and the empty
bijection. Brauer's first main theorem shows that the local set is also
empty. Moreover, $N_K(D_0)<K$, since $O_p(K)\leq Z(K)$. We may
therefore assume that $\Irr_0(K\mid D_0)$ is non-empty. In particular,
$D_0$ is a block defect group.

We have $K=EZ(K)$, where $E=K'$ is perfect and $E/Z(E)$ has the
same simple factors as $K/Z(K)$. The block and defect-group
correspondences for this central product give
$D_0=D_EZ(K)_p$, where $D_E=D_0\cap E$ is a defect group of the
corresponding block of $E$. Moreover, $D_E$ is non-central.
Apply Proposition~\ref{r6:6}
to obtain $M_E$ and $\Omega_E$, and set $M=M_EZ(K)$. Since
$Z(E)\leq M_E$, we have $M\cap E=M_E$, and hence $M<K$.

Each $\theta$ has a unique expression as a central product of
characters with compatible central characters,
\[
 \theta=\theta_E\cdot\nu,\qquad \nu\in\Irr(Z(K)).
\]
Define
\[
 \Omega(\theta)=\Omega_E(\theta_E)\cdot\nu.
\]
Since $\Omega_E$ preserves central characters, this defines a bijection.
Computing blocks, defect groups, and heights on $E\times Z(K)$ and
then passing through the central kernel gives the stated properties.

For the strong relation, set
\[
 G=A_{\theta_E^{\mathcal H}},\qquad H_E=M_E N_G(D_E).
\]
Since $K=EZ(K)$, we have $K\leq G_{\theta_E}$ and
$K\trianglelefteq G$. Moreover,
\[
 N_K(D_E)=N_E(D_E)Z(K),\qquad K\cap H_E=M_EZ(K)=M.
\]
Apply \cite[Theorem~2.7]{ChenCT} to the normal intermediate subgroup $K$.
This theorem uses the original pair of projective representations
affording the relation over $E$. The associated map $\tau_K$ preserves central
characters on $Z(K)\leq C_G(E)$, so
\[
 \tau_K(\theta_E\cdot\nu)=\Omega_E(\theta_E)\cdot\nu.
\]
Consequently, all mixed comparison functions and the block conditions
on every larger intermediate subgroup are preserved.

Finally, $D_0=D_EZ(K)_p$ and $D_E=D_0\cap E$ imply
\[
 N_A(D_0)=N_A(D_E).
\]
Thus $H=MN_A(D_0)=M_EN_A(D_E)$ and $H_E=H\cap G$. If
$\varphi^h=\varphi^\sigma$, restriction to $M_E$ and equivariance of
$\Omega_E$ give $\theta_E^h=\theta_E^\sigma$, so $h\in G$.
Similarly, $A_{\theta^{\mathcal H}}\leq G$. Thus the overgroups obtained from the homogeneous transfer theorem are
precisely
$A_{\theta^{\mathcal H}}$ and $H_{\varphi^{\mathcal H}}$. This
completes the proof.
\end{proof}

\begin{remark}\label{r6:8}
The hypothesis on simple groups in Theorem~\ref{r6:7} is exactly
Definition~\ref{r6:1}. The corresponding results for arbitrary overgroups, direct products,
and diagonal central quotients follow from this condition. The ordinary block arguments are drawn from \cite{NS14,Sp13,Sp17},
and the Galois comparison functions are constructed as in
\cite{NSV20}.
The block conditions for those constructions are established by
\cite[Lemma~1.10]{ChenCT} and Lemma~\ref{r6:5}. The mixed comparison
functions on the individual universal covering groups remain part of
the iAMN condition; they do not follow merely from the ordinary
inductive Alperin--McKay condition.
\end{remark}

\section{Proof of the reduction theorem}\label{sec:reduction}

We prove the reduction theorem by the group-theoretic induction of
\cite[Section~7]{NS14}. Throughout this section, $\geq_b$ denotes the
relation of \cite[Definition~1.1]{ChenCT}. Thus the same pair of
associated projective representations must satisfy both the block
conditions and the required equalities between the Galois comparison
functions.

\subsection{The inductive assertion and the transfer results}\label{r7:sub:1}
Fix the prime $p$ and the group $\mathcal H$, and set
\[
 m(V)=|V:Z(V)|
\]
for every finite group $V$. Let $\mathsf R(V)$ denote the following
assertion: for every finite group $B$ containing $V$ as a normal subgroup
and every $p$-subgroup $Q\leq V$, there is an
$N_B(Q)\times\mathcal H$-equivariant bijection
\begin{equation}\label{r7:eq:1}
 \Pi_{V,Q}\colon\Irr_0(V\mid Q)\longrightarrow
                 \Irr_0(N_V(Q)\mid Q)
\end{equation}
such that
\begin{equation}\label{r7:eq:2}
 \bl(\Pi_{V,Q}(\rho))^V=\bl(\rho)
\end{equation}
and, writing $\rho'=\Pi_{V,Q}(\rho)$,
\begin{equation}\label{r7:eq:3}
 (B_{\rho^{\mathcal H}},V,\rho)_{\mathcal H}
 \geq_b
 \bigl(N_B(Q)_{(\rho')^{\mathcal H}},N_V(Q),\rho'\bigr)_{\mathcal H}.
\end{equation}
The assertion allows an arbitrary finite ambient group $B$, including
the finite semidirect products that arise in the centralization
construction.

We first recall the four transfer results established or recorded
earlier that will be used in the proof.

\smallskip
\noindent\emph{The central-defect DGN correspondence.}
Corollary~\ref{r3:2} provides the required correspondence on all
lower-character fibers when $K\leq V$ are normal in an ambient group,
$V=KD$, and $K\cap D\leq Z(V)$. The correspondence takes values in
$\Irr_0(N_V(D)\mid D)$. This result from the related DGN paper is
recorded in Section~\ref{sec:inputs}.

\smallskip
\noindent\emph{Semilinear centralizing descent.}
Proposition~\ref{r4:18} assumes that $\mathsf R(W)$ holds for every
$W$ with $m(W)<n$, with no restriction on the finite ambient group.
For a $T$-invariant character $\zeta\in\Irr(K)$ and $|T:K|<n$, it
provides a bijection
\[
 \Irr_0(T\mid D,\zeta^{\mathcal H})\longrightarrow
 \Irr_0(KN_T(D)\mid D,\zeta^{\mathcal H})
\]
with the stated fiber, block, equivariance, and block
$\mathcal H$-triple properties. The only group to which induction is
applied in that construction is a quotient satisfying
\[
 m(\widehat T/K_0)\leq |T:K|<n.
\]
The external bound $n$ is kept fixed when $T$ is an inertia subgroup;
in particular, the construction does not require
$|T:K|<m(T)$.

\smallskip
\noindent\emph{Gluing correspondences.}
The first gluing result, \cite[Theorem~3.2]{ChenCT}, lifts a block
$\mathcal H$-triple correspondence from a normal subgroup to a normal
intermediate group with a prescribed upper defect group. The second,
\cite[Theorem~3.3]{ChenCT}, combines compatible centralizing correspondences
on the lower-character inertia groups. With the notation of the latter
result, let $u_\zeta$ be the map on the $\zeta$-fiber and let $\gamma$
be the Clifford correspondent of $\chi$. The resulting map is
\begin{equation}\label{r7:eq:4}
 \begin{split}
 u\colon\Irr_0(X\mid D)&\longrightarrow\Irr_0(Y\mid D),
 \qquad Y=KN_X(D),\\
 u(\chi)&=u_\zeta(\gamma)^Y.
 \end{split}
\end{equation}
The projective representations constructed in
\cite[Lemma~3.1]{ChenCT} satisfy both the required equalities between
the mixed comparison functions and the block conditions on every
intermediate group.
The second gluing result requires no equality between the indices of
the two inducing subgroups.

\smallskip
\noindent\emph{Transfer from quasisimple components.}
Theorem~\ref{r6:7} applies to a central product $K$ of quasisimple
groups whose simple quotients satisfy Definition~\ref{r6:1}. For a
non-central radical subgroup $D_0\leq K$, it provides an invariant
proper subgroup $N_K(D_0)\leq M<K$ and a block correspondence from
$K$ to $M$ on all height-zero characters with specified defect group
$D_0$. The resulting block $\mathcal H$-triple relation is realized
throughout by the pair constructed in that theorem.

\begin{theorem}\label{r7:1}
Suppose that every non-abelian finite simple group satisfies the inductive
condition in Definition~\ref{r6:1} for the prime $p$. Then
$\mathsf R(X)$ holds for every finite group $X$. In particular, the
block-bijection form of the Alperin--McKay--Navarro conjecture holds for
all finite groups at $p$.
\end{theorem}

\subsection{Preparatory lemmas}\label{r7:sub:2}

\begin{lemma}\label{r7:2}
Let $X\unlhd A$, and let $D\leq X$ be a $p$-subgroup. In proving
$\mathsf R(X)$, we may assume that $D$ is a defect group of a block of
$X$ and that $A=XN_A(D)$. If $N_X(D)=X$, the identity map has all the
required properties.
\end{lemma}

\begin{proof}
Every block contains a height-zero character. If $D$ is not a defect
group of any block of $X$, then $\Irr_0(X\mid D)$ is empty.
The set $\Irr_0(N_X(D)\mid D)$ is also empty, since Brauer's first
main theorem would otherwise yield a block of $X$ with defect
group $D$.

Set $A_1=XN_A(D)$. Then $N_{A_1}(D)=N_A(D)$. If
$\chi\in\Irr_0(X\mid D)$ and $a\in A_{\chi^{\mathcal H}}$, choose
$\sigma\in\mathcal H$ such that $\chi^a=\chi^\sigma$.
By \cite[Lemma~1.4]{ChenCT}, the blocks of $\chi$ and $\chi^\sigma$ have
the same defect groups. Hence $D^a$ is $X$-conjugate to $D$, and
therefore
\begin{equation}\label{r7:eq:5}
 A_{\chi^{\mathcal H}}\leq A_1.
\end{equation}
Consequently, a block $\mathcal H$-triple relation obtained in $A_1$
already has the required ambient group. If $N_X(D)=X$, then
$X\leq N_A(D)$, so
\eqref{r7:eq:5} gives $A_{\chi^{\mathcal H}}\leq N_A(D)$.
The two triples coincide, and choosing the same associated projective
representation on both sides gives the identity correspondence with
all the required properties.
\end{proof}

\begin{lemma}\label{r7:3}
Let $X\unlhd A$, let $D\leq X$ be a $p$-subgroup, and assume that
$\mathsf R(W)$ holds for every finite group $W$ with $m(W)<m(X)$.
If $K\unlhd A$, $K\leq X$, and $|X:K|<m(X)$, then there is an
$N_A(D)\times\mathcal H$-equivariant bijection
\[
 \Upsilon_{D,K}\colon\Irr_0(X\mid D)\longrightarrow
                         \Irr_0(KN_X(D)\mid D)
\]
preserving block induction. For $\chi'=\Upsilon_{D,K}(\chi)$, it
satisfies
\[
 (A_{\chi^{\mathcal H}},X,\chi)_{\mathcal H}
 \geq_b
 \bigl((KN_A(D))_{(\chi')^{\mathcal H}},KN_X(D),\chi'\bigr)_{\mathcal H}.
\]
\end{lemma}

\begin{proof}
Set $D_0=K\cap D$. For each $D$-invariant character
$\zeta\in\Irr_0(K\mid D_0)$, set $T_\zeta=X_\zeta$.
Then $D\leq T_\zeta$ and $T_\zeta\unlhd A_{\zeta^{\mathcal H}}$.
Indeed, if $\zeta^a=\zeta^\sigma$, then $\zeta^\sigma$ and $\zeta$
have the same inertia group in $X$, and hence $a$ normalizes $T_\zeta$.

Apply Proposition~\ref{r4:18} with the fixed bound $n=m(X)$.
Since
\[
 |T_\zeta:K|\leq |X:K|<n,
\]
its hypotheses hold for $T_\zeta\unlhd A_{\zeta^{\mathcal H}}$.
The quotient used in its centralizing construction has central index
at most $|T_\zeta:K|$; no comparison between $|T_\zeta:K|$ and
$m(T_\zeta)$ is needed. The same external bound is used throughout
the central extension,
the Gallagher correspondence, and the block descent.

Choose the data of Proposition~\ref{r4:18} for representatives of the
$N_A(D)\times\mathcal H$-orbits on these lower characters, and
transport them by conjugation and Galois action. Equivariance under
the mixed stabilizer of each representative ensures that the resulting
fiber maps are independent of the transporting element. They therefore
provide the compatible data required by \cite[Theorem~3.3]{ChenCT}.

By \cite[Proposition~2.5(a),(f)]{NS14}, every
$\chi\in\Irr_0(X\mid D)$ has such a constituent whose Clifford
correspondent has defect group $D$; any two such choices are related
by an element of $N_X(D)$. Applying \cite[Theorem~3.3]{ChenCT} now
gives the map in
\eqref{r7:eq:4} with all the asserted properties.
\end{proof}

\subsection{The central-index induction}\label{r7:sub:3}

\begin{proof}[Proof of Theorem~\ref{r7:1}]
We prove $\mathsf R(X)$ by strong induction on $m(X)$, treating all
finite ambient groups simultaneously. If $m(X)=1$, then $X$ is
abelian, and the identity case of Lemma~\ref{r7:2} applies.

Assume that $\mathsf R(W)$ holds for every finite group $W$ of smaller
central index. Fix $X\unlhd A$ and a $p$-subgroup $D\leq X$.
By Lemma~\ref{r7:2}, we may assume that $D$ is a defect group of a
block of $X$, that $A=XN_A(D)$, and that $N_X(D)<X$.
In particular, $X$ is non-abelian. A normal $p$-subgroup is contained
in every block defect group, so
\begin{equation}\label{r7:eq:6}
 O_p(X)\leq D,\qquad Z(X)_p\leq D.
\end{equation}
We distinguish three cases.

\medskip
\noindent\emph{Case 1: a normal subgroup gives a proper local subgroup.}
Suppose that $K\unlhd A$ satisfies
\[
 Z(X)<K\leq X,\qquad Y=KN_X(D)<X.
\]
Then $|X:K|<m(X)$, and Lemma~\ref{r7:3} provides
$\Upsilon_{D,K}$. Set $E=KN_A(D)$. Since $X\unlhd A$,
\begin{equation}\label{r7:eq:7}
 X\cap E=KN_X(D)=Y,
\end{equation}
so $Y\unlhd E$. Moreover, $Z(X)\leq Z(Y)$ and $Y<X$, whence
\begin{equation}\label{r7:eq:8}
 m(Y)\leq |Y:Z(X)|<|X:Z(X)|=m(X).
\end{equation}
The inclusions $N_A(D)\leq E$ and $N_X(D)\leq Y$ give
\begin{equation}\label{r7:eq:9}
 N_E(D)=N_A(D),\qquad N_Y(D)=N_X(D).
\end{equation}
Apply the inductive hypothesis to $Y\unlhd E$ to obtain
$\Pi_{Y,D}$, and define
\[
 \Pi_{X,D}=\Pi_{Y,D}\circ\Upsilon_{D,K}.
\]
For $\chi'=\Upsilon_{D,K}(\chi)$, the two block
$\mathcal H$-triple relations have the common triple
$(E_{(\chi')^{\mathcal H}},Y,\chi')_{\mathcal H}$.
Their composition is therefore covered by \cite[Lemma~1.6]{ChenCT}.
Together with equivariance and the transitivity of block induction,
this gives \eqref{r7:eq:1}--\eqref{r7:eq:3}.

For the remainder of the proof, we may therefore assume that
\begin{equation}\label{r7:eq:10}
 X=KN_X(D)\quad\text{whenever }K\unlhd A
                    \text{ and }Z(X)<K\leq X.
\end{equation}

\medskip
\noindent\emph{Case 2: a normal subgroup has central defect intersection.}
Suppose that
\[
 K\unlhd A,\qquad Z(X)<K\leq X,\qquad K\cap D\leq Z(X).
\]
By \eqref{r7:eq:10} and $A=XN_A(D)$,
\begin{equation}\label{r7:eq:11}
 X=KN_X(D),\qquad A=KN_A(D).
\end{equation}
Set $V=KD$ and $H=N_A(D)$. Since $H$ normalizes both $K$ and $D$,
and $K$ normalizes $KD$,
the factorization $A=KH$ implies that $V\unlhd A$.
Furthermore,
\[
 V/K\text{ is a }p\text{-group},\qquad
 K\cap D\leq Z(V),\qquad A=VH.
\]
Corollary~\ref{r3:2} supplies the required correspondence between all
height-zero characters of $V$ and $N_V(D)=V\cap H$ with defect
group $D$. Apply \cite[Theorem~3.2]{ChenCT} with normal subgroup $V$,
defect group $D$, normal intermediate group $J=X$, and prescribed
upper defect group $Q=D$. Here $Q\cap V=D$ and
$H=(V\cap H)N_A(D)$, as required. This gives
\[
 \Irr_0(X\mid D)\longrightarrow
 \Irr_0(X\cap H\mid D)=\Irr_0(N_X(D)\mid D),
\]
together with \eqref{r7:eq:2} and \eqref{r7:eq:3}.
This completes Case~2 without a further application of induction.

For the remaining case, we may assume, in addition to
\eqref{r7:eq:10}, that
\begin{equation}\label{r7:eq:12}
 \text{no }K\unlhd A\text{ satisfies }
 Z(X)<K\leq X\text{ and }K\cap D\leq Z(X).
\end{equation}

\medskip
\noindent\emph{Case 3: the layer gives a proper local subgroup.}
We first prove that $F(X)\leq Z(X)$. The subgroup
$O_{p'}(X)Z(X)$ is characteristic in $X$, hence normal in $A$.
Its unique Sylow $p$-subgroup is $Z(X)_p$, which is central; by
\eqref{r7:eq:6}, its intersection with $D$ is therefore $Z(X)_p$.
Equation~\eqref{r7:eq:12} yields
\begin{equation}\label{r7:eq:13}
 O_{p'}(X)\leq Z(X).
\end{equation}
If $O_p(X)\nleq Z(X)$, then $K=O_p(X)Z(X)$ satisfies the
hypotheses in \eqref{r7:eq:10}. On the other hand,
\eqref{r7:eq:6} gives $K\leq N_X(D)$, so
\[
 X=KN_X(D)=N_X(D),
\]
a contradiction. Hence $O_p(X)\leq Z(X)$. Together with
\eqref{r7:eq:13} and the nilpotence of $F(X)$, this implies
\begin{equation}\label{r7:eq:14}
 F(X)\leq Z(X).
\end{equation}

Let $K=E(X)$ be the layer of $X$. It is a central product of
quasisimple components and is characteristic in $X$, so $K\unlhd A$.
The generalized Fitting subgroup $F^*(X)=E(X)F(X)$ contains its own
centralizer in $X$. If $K\leq Z(X)$, then \eqref{r7:eq:14} would
give $F^*(X)\leq Z(X)$, and hence
\[
 X=C_X(F^*(X))\leq F^*(X)\leq Z(X),
\]
contrary to our assumption. Thus $K\nleq Z(X)$. Also, $Z(K)$ is
an abelian normal subgroup of $X$, so
\begin{equation}\label{r7:eq:15}
 Z(K)\leq F(X)\leq Z(X).
\end{equation}

Set $D_0=D\cap K$. To apply \eqref{r7:eq:12}, consider the normal
subgroup $KZ(X)$, which properly contains $Z(X)$. Since
$Z(X)_p\leq D$,
\begin{equation}\label{r7:eq:16}
 D\cap KZ(X)=(D\cap K)Z(X)_p.
\end{equation}
Indeed, $KZ(X)/K$ is a quotient of $Z(X)$ and has Sylow
$p$-subgroup $KZ(X)_p/K$. The image of $D\cap KZ(X)$ is a
$p$-group containing this Sylow subgroup, and hence equals it.
The kernel of this map is $D\cap K$, proving \eqref{r7:eq:16}.
If $D_0\leq Z(K)$, then \eqref{r7:eq:15} and \eqref{r7:eq:16}
would imply $D\cap KZ(X)\leq Z(X)$, contrary to
\eqref{r7:eq:12}. Consequently,
\begin{equation}\label{r7:eq:17}
 D_0\nleq Z(K).
\end{equation}

Choose $\chi\in\Irr_0(X\mid D)$. By
\cite[Proposition~2.5(a),(f)]{NS14}, its restriction to $K$ has a
constituent $\zeta$ whose Clifford correspondent has defect group
$D$. In particular, $D_0$ is a defect group of $\bl(\zeta)$ and is
therefore radical. To see this, its Brauer correspondent in
$N_K(D_0)$ has defect group $D_0$, so
\[
 O_p(N_K(D_0))\leq D_0.
\]
The reverse inclusion follows from $D_0\unlhd N_K(D_0)$.

Theorem~\ref{r6:7} now provides an $N_A(D_0)$-invariant subgroup
$M$ such that
\begin{equation}\label{r7:eq:18}
 N_K(D_0)\leq M<K,
\end{equation}
together with the stated block $\mathcal H$-triple correspondence on
all lower-character fibers. Since $KZ(X)>Z(X)$, equation
\eqref{r7:eq:10} gives
\[
 X=KZ(X)N_X(D)=KN_X(D).
\]
Normality of $K$ in $A$ implies $N_A(D)\leq N_A(D_0)$, whence
\begin{equation}\label{r7:eq:19}
 A=KN_A(D)=KN_A(D_0).
\end{equation}
Set
\[
 H=MN_A(D_0),\qquad Y=X\cap H=MN_X(D_0).
\]
The factorization $A=KN_A(D)$ also yields
\[
 N_A(D_0)=N_K(D_0)N_A(D),
\]
since $N_A(D)$ normalizes $D_0$. Thus
$H=MN_A(D)$ and $Y=MN_X(D)$. These factorizations allow us to apply
\cite[Theorem~3.2]{ChenCT} to the component correspondence.
By \eqref{r7:eq:18},
\begin{equation}\label{r7:eq:20}
 K\cap H=MN_K(D_0)=M,\qquad K\cap Y=M.
\end{equation}
It follows that $Y<X$: otherwise $K\leq Y$ would imply $K=M$.
Moreover, \eqref{r7:eq:19} gives $A=KH$, and
$Y=X\cap H\unlhd H$. All the hypotheses of
\cite[Theorem~3.2]{ChenCT} hold with normal subgroups $K,M$, normal
intermediate group $X$, and specified defect group $D$:
\[
 H=MN_A(D_0),\qquad D\leq N_X(D_0)\leq Y,
 \qquad D\cap K=D_0.
\]
We obtain a block bijection
\[
 \Lambda_X\colon\Irr_0(X\mid D)\longrightarrow\Irr_0(Y\mid D)
\]
satisfying the required block $\mathcal H$-triple relation.

Finally, $Z(X)\leq C_X(D_0)\leq Y$, so
\begin{equation}\label{r7:eq:21}
 m(Y)\leq |Y:Z(X)|<m(X).
\end{equation}
The inclusions $N_A(D)\leq H$ and $N_X(D)\leq Y$ also give
\begin{equation}\label{r7:eq:22}
 N_H(D)=N_A(D),\qquad N_Y(D)=N_X(D).
\end{equation}
Apply induction to $Y\unlhd H$, and define
$\Pi_{X,D}=\Pi_{Y,D}\circ\Lambda_X$.
For $\chi'=\Lambda_X(\chi)$, the two block $\mathcal H$-triple relations have the common middle
triple $(H_{(\chi')^{\mathcal H}},Y,\chi')_{\mathcal H}$.
By \cite[Lemma~1.6]{ChenCT} and \eqref{r7:eq:22}, composition gives
\eqref{r7:eq:3}. The composite also preserves block induction,
the prescribed defect group, and equivariance.

The three cases exhaust the possibilities, completing the induction.
\end{proof}

\subsection{The blockwise conclusion}\label{r7:sub:6}

\begin{corollary}\label{r7:4}
Assume the hypothesis of Theorem~\ref{r7:1}. Let $b$ be a $p$-block
of a finite group $G$, let $D$ be a defect group of $b$, and let $c$
be its Brauer correspondent in $N_G(D)$. Then there is an
$N_{\Aut(G)}(D)_b\times\mathcal H_b$-equivariant bijection
\[
 \Irr_0(b)\longrightarrow\Irr_0(c).
\]
It is the restriction of an
$N_{\Aut(G)}(D)\times\mathcal H$-equivariant bijection on all
height-zero characters with defect group $D$. In particular, for
every $\sigma\in\mathcal H$,
\[
 |\{\chi\in\Irr_0(b):\chi^\sigma=\chi\}|
 =|\{\psi\in\Irr_0(c):\psi^\sigma=\psi\}|.
\]
\end{corollary}

\begin{proof}
Apply $\mathsf R(G)$ with ambient group $G\rtimes\Aut(G)$.
Equation~\eqref{r7:eq:2} and the uniqueness assertion in Brauer's
first main theorem show that the global bijection restricts to a bijection from
$\Irr_0(b)$ onto $\Irr_0(c)$. The equivariance of the global bijection
gives the stated equivariance of its restriction. If $\sigma$ stabilizes
$b$, then it also stabilizes $c$, and
the restricted bijection identifies the two fixed-point sets.
If $\sigma$ does not stabilize $b$, it does not stabilize $c$ either,
and both fixed-point sets are empty.
\end{proof}

\begin{remark}\label{r7:5}
The proof uses only induction on $m(X)$. Each application to a proper
subgroup is justified by a strict inequality for its central index.
In Proposition~\ref{r4:18}, induction is applied to a quotient of a
central extension whose central index is strictly below the same
external bound. Induction of projective representations and the
resulting gluing
construction, without requiring equal indices for the inducing
subgroups, are treated in \cite[Lemma~3.1]{ChenCT} and
\cite[Theorem~3.3]{ChenCT}; the block conditions
for the projective representations constructed from the components are established in
Section~\ref{sec:components}. The central-defect DGN correspondence
is the result from the related DGN paper recorded in
Section~\ref{sec:inputs},
and semilinear centralizing descent is proved in
Section~\ref{sec:centralization}. The four transfer results recalled
above are therefore available
under the hypothesis of Theorem~\ref{r7:1}, namely the condition on
individual simple groups in Definition~\ref{r6:1}.

A version for a prescribed collection of simple groups must assume
that the relevant non-abelian simple sections of $X$ satisfy this
condition, or work in a class closed under the subgroups, quotients,
and finite central extensions used in the proof. Here a simple section
means a simple quotient of a subgroup. An assumption concerning only
the non-abelian composition factors of $X$ does not, in general,
justify the subgroup induction above. In either formulation, the
finite ambient group remains arbitrary.
\end{remark}

\subsection{Arithmetic and structural consequences}\label{sec:consequences}

We record several consequences of the equivariant bijections above.
The arithmetic and local character-theoretic results used here are
established in the cited literature. The reduction theorem makes
them available under the iAMN hypothesis on simple groups; the
results that already have independent proofs are identified below.

\begin{corollary}\label{red:classical-consequences}
Assume that every non-abelian finite simple group satisfies iAMN
for $p$. Let $G$ be a finite group and $P\in\Syl_p(G)$.
Then the Alperin--McKay conjecture holds for every $p$-block of $G$,
and there is an $\mathcal H$-equivariant bijection
\[
 \Irr_{p'}(G)\longrightarrow\Irr_{p'}(N_G(P)),
\]
where $\Irr_{p'}$ denotes the set of irreducible characters of
degree prime to $p$.

If $p$ is odd and divides $|G|$, let $\tau\in
\Gal(\mathbb Q(\zeta_{|G|})/\mathbb Q)$ fix all $p'$-roots of
unity in this field and have order $p-1$. Then
\[
 N_G(P)=P
 \quad\Longleftrightarrow\quad
 \Irr_{p'}(G)^\tau=\{1_G\}.
\]
If $p=2$, let $\rho\in\mathcal H_2$ fix all $2$-power roots of
unity and square all odd-order roots of unity. Then
\[
 N_G(P)=P
 \quad\Longleftrightarrow\quad
 \Irr_{2'}(G)^\rho=\Irr_{2'}(G).
\]
\end{corollary}

\begin{proof}
Corollary~\ref{r7:4} gives the Alperin--McKay equality by taking
the identity Galois automorphism. The $p'$-degree characters are
exactly the height-zero characters in blocks of full defect.
Thus Theorem~\ref{r7:1}, with specified defect group $P$, gives
the asserted equivariant bijection, and hence the Galois--McKay
fixed-point equalities. The self-normalizing criteria now follow
from \cite[Theorems~5.2 and~5.3]{Nav04}.
\end{proof}

For a positive integer $n$, write $\mathbb Q_n=\mathbb Q(\zeta_n)$,
where $\zeta_n$ is a primitive $n$th root of unity. The conductor
$f(\chi)$ of a character $\chi$ is the least positive integer $f$
such that $\mathbb Q(\chi)\subseteq\mathbb Q_f$.
Its \emph{$p$-rationality level} is
$\operatorname{lev}_p(\chi)=v_p(f(\chi))$.
A character is $p$-rational if this level is zero, and almost
$p$-rational if it is at most one.

\begin{corollary}\label{red:arithmetic-consequences}
Assume that every non-abelian finite simple group satisfies iAMN
for $p$. Let $b$ be a $p$-block of $G$ with defect group $D$, and
let $c$ be its Brauer correspondent in $N_G(D)$.
Set $L=\mathbb Q_{|G|}$, let $\Gamma_b$ be the image of
$\mathcal H_b$ in $\Gal(L/\mathbb Q)$, and set $F=L^{\Gamma_b}$.
There is a bijection $\Omega:\Irr_0(b)\longrightarrow\Irr_0(c)$
such that, for every $\chi\in\Irr_0(b)$,
\[
 F\mathbb Q(\chi)=F\mathbb Q(\Omega(\chi)),
 \qquad
 \operatorname{lev}_p(\chi)
 =\operatorname{lev}_p(\Omega(\chi)).
\]
In particular, the two blocks have the same number of height-zero
characters of each $p$-rationality level.
\end{corollary}

\begin{proof}
Take the bijection in Corollary~\ref{r7:4}. Equivariance and
injectivity give $\Gamma_{b,\chi}=\Gamma_{b,\Omega(\chi)}$.
The fixed fields of these stabilizers are respectively
$F\mathbb Q(\chi)$ and $F\mathbb Q(\Omega(\chi))$, proving
the first equality.

Write $|G|=p^a m$, where $(p,m)=1$. Every element of
$I=\Gal(L/\mathbb Q_m)$ lifts to an element of $\mathcal H$
fixing all $p'$-roots of unity. Such an element fixes the Brauer
characters and hence every $p$-block, so $I\leq\Gamma_b$.
For $0\leq r\leq a$, set
$I_r=\Gal(L/\mathbb Q_{p^r m})$.
A character with values in $L$ has $p$-rationality level at most
$r$ precisely when it is fixed by $I_r$.
The bijection therefore preserves every level, as also explained
in \cite[Section~7.1]{HS26}.
\end{proof}

The field equality in Corollary~\ref{red:arithmetic-consequences}
is an equality after adjoining $F$; it does not assert equality of
the absolute fields of values. Nor does it assert that restriction
to $N_G(D)$ preserves the $p$-rationality level of $\chi$.
The latter property is a separate conjecture in \cite{HS26};
see its Theorem~7.2 for the relation with AMN.

Let $B_0(G)$ denote the principal $p$-block. Let $\sigma_p$
fix all $p'$-roots of unity and send every $p$-power root of unity
to its $(1+p)$th power, and define
\[
 k_{0,\sigma_p}(B_0(G))
 =\bigl|\{\chi\in\Irr_0(B_0(G))\mid
                  \chi^{\sigma_p}=\chi\}\bigr|.
\]
For odd $p$, being $\sigma_p$-fixed is equivalent to being almost
$p$-rational. For $p=2$ this equivalence holds for odd-degree
irreducible characters, and hence for the height-zero characters in
$B_0(G)$; see the introduction to \cite{MMRSV26}.

\begin{corollary}\label{red:sylow-consequences}
Assume that every non-abelian finite simple group satisfies iAMN
for $p$. Let $P\in\Syl_p(G)$ with $P\ne1$, and write
$k=k_{0,\sigma_p}(B_0(G))$. Then:
\begin{enumerate}
\item $k\geq\lceil2\sqrt{p-1}\rceil$, and equality implies that
$P$ is cyclic.
\item If $p\in\{2,3\}$, then $P$ is cyclic if and only if $k=p$.
\item If $p=2$, then $|P:\Phi(P)|=4$ if and only if $k=4$.
\item If $p=3$, then $|P:\Phi(P)|=9$ if and only if $k\in\{6,9\}$.
\end{enumerate}
Here $\Phi(P)$ denotes the Frattini subgroup of $P$; the
conditions on $|P:\Phi(P)|$ in \textup{(iii)} and \textup{(iv)}
mean that the minimum number of generators of $P$ is two.
\end{corollary}

\begin{proof}
Corollary~\ref{r7:4} and Brauer's third main theorem give
\[
 k_{0,\sigma_p}(B_0(G))
 =k_{0,\sigma_p}(B_0(N_G(P))).
\]
The local principal-block calculation in \cite[Section~7.2]{MMRSV26}
identifies this number with
\[
 k\bigl(N_G(P)/\Phi(P)O_{p'}(N_G(P))\bigr),
\]
where $k(Y)=|\Irr(Y)|$. The same section explains how the local
bound and its equality case yield \textup{(i)}.
Assertion~\textup{(ii)} is \cite[Theorem~A]{RSV20}.
The implications from AMN to \textup{(iii)} and \textup{(iv)}
are recorded in the introduction to \cite{NRSV21}; see also
\cite[Introduction]{KMRS26} for \textup{(iv)}.
\end{proof}

The conclusions in \textup{(i)}, \textup{(ii)} and \textup{(iii)}
of Corollary~\ref{red:sylow-consequences} have independent proofs
in \cite{MMRSV26}, \cite{RSV20} and \cite{NRSV21}, respectively.
For \textup{(iv)}, \cite[Theorem~A]{KMRS26} proves independently
that $k\in\{6,9\}$ implies $|P:\Phi(P)|=9$.
We include these statements to make explicit the arithmetic and
group-theoretic consequences of the reduction, without presenting
the previously established conclusions as new results.

\section{Examples of the inductive condition}\label{sec:verification}

We first establish a criterion for groups whose universal covering
group has only inner automorphisms. The criterion includes the block
condition on every intermediate subgroup. Combined with the
Galois-equivariant correspondence for cyclic defect groups, it verifies
Definition~\ref{r6:1} for several sporadic groups at specified primes.
We then treat Suzuki and small Ree groups in defining characteristic,
using Johansson's Galois-compatible extensions together with the block
criterion of \cite[Theorem~4.4]{NS14}.

\subsection{A criterion for verification}\label{r8:sub:1}
For a centerless group $X$, write $c_x\in\Inn(X)$ for the
automorphism $u\mapsto xux^{-1}$. The map $x\mapsto c_x$ is then
an isomorphism $X\longrightarrow\Inn(X)$. We use the following
multiplication in the semidirect products below:
\[
 (x,c_n)(y,c_m)=(xnyn^{-1},c_{nm}).
\]
Characters of a direct factor are identified with characters of the
corresponding subgroup.

\begin{lemma}\label{r8:1}
Let $X$ be a centerless finite group, let $D\leq X$ be a $p$-subgroup,
and suppose that
\[
 N=N_X(D)\leq L\leq X,\qquad \Gamma=\Inn(X)_D.
\]
If $\chi\in\Irr_0(X\mid D)$ and
$\varphi\in\Irr_0(L\mid D)$ satisfy
\[
 \mathcal H_\chi=\mathcal H_\varphi,
 \qquad \bl(\varphi)^X=\bl(\chi),
\]
then
\begin{equation}\label{r8:eq:1}
 (X\rtimes\Gamma,X,\chi)_{\mathcal H}
 \geq_b
 (L\rtimes\Gamma,L,\varphi)_{\mathcal H}.
\end{equation}
This relation is realized by ordinary representations with trivial
factor sets and trivial mixed comparison functions.
\end{lemma}

\begin{proof}
Since $X$ is centerless, $\Gamma=\{c_n:n\in N\}$. The inclusion
$N\leq L$ implies that $\Gamma$ acts by inner automorphisms on
both $X$ and $L$, so it fixes $\chi$ and $\varphi$. The map
\begin{equation}\label{r8:eq:2}
 X\rtimes\Gamma\longrightarrow X\times N,
 \qquad (x,c_n)\longmapsto(xn,n),
\end{equation}
is an isomorphism. Indeed, the first coordinate of the image of
$(x,c_n)(y,c_m)$ is $xnym=(xn)(ym)$, and the inverse map sends
$(u,n)$ to $(un^{-1},c_n)$. This isomorphism sends $X$ to $X\times1$ and $L\rtimes\Gamma$
to $L\times N$, and is the identity on the first direct factor.
We verify the required relation in these direct product coordinates.

Set $G=X\times N$ and $H=L\times N$. Then
$G=(X\times1)H$ and
\[
 C_G(X\times1)=1\times N\leq H.
\]
Both $G$ and $H$ act by inner automorphisms on their respective
normal subgroups $X\times1$ and $L\times1$. Consequently,
$G_\chi=G$, $H_\varphi=H$, and
\[
 (H\times\mathcal H)_\chi
 =H\times\mathcal H_\chi
 =H\times\mathcal H_\varphi
 =(H\times\mathcal H)_\varphi.
\]
Choose representations $R$ and $R'$ over a sufficiently large finite
cyclotomic field affording $\chi$ and $\varphi$, respectively, and set
\[
 \mathcal P(x,n)=R(x),\qquad
 \mathcal P'(l,n)=R'(l).
\]
These ordinary representations of $G$ and $H$ restrict to
representations affording the prescribed characters of the normal
subgroups. Their factor sets are trivial, and both send every
element of $1\times N$ to the identity matrix.

Let $a=((l,n),\sigma)\in(H\times\mathcal H)_\chi$.
Since $\chi^\sigma=\chi$, there is an invertible matrix $T_\sigma$
such that
\[
 R(x)^\sigma=T_\sigma R(x)T_\sigma^{-1}\qquad(x\in X).
\]
For $(x,j)\in G$, we then have
\[
 \mathcal P\bigl((l,n)(x,j)(l,n)^{-1}\bigr)^\sigma
 =R(l)^\sigma T_\sigma\mathcal P(x,j)T_\sigma^{-1}
       \bigl(R(l)^\sigma\bigr)^{-1}.
\]
Thus the comparison function for $\mathcal P$ is $\mu_a=1$.
Since $\sigma\in\mathcal H_\varphi$, the same calculation for
$R'$ gives $\mu'_a=1$. The comparison functions therefore agree,
as required; equality of the comparison matrices is not needed.

We now verify the block condition for the same pair of
representations. Let $X\times1\leq W\leq G$. Multiplication by
elements of $X\times1$ shows that every second coordinate occurring
in $W$ also occurs in $W\cap(1\times N)$. Consequently, there is a
unique subgroup $J\leq N$ such that $W=X\times J$, and then
$W\cap H=L\times J$. The pair $\mathcal P,\mathcal P'$ determines
the correspondence
\begin{equation}\label{r8:eq:3}
 \chi\otimes\eta\longmapsto\varphi\otimes\eta
 \qquad(\eta\in\Irr(J)),
\end{equation}
since both representations are trivial on the second direct factor.

Write $B=\bl(\chi)$, $b=\bl(\varphi)$, and $c=\bl(\eta)$.
The corresponding characters in \eqref{r8:eq:3} lie in the
direct product blocks $B\otimes c$ and $b\otimes c$, respectively.
We claim that
\begin{equation}\label{r8:eq:4}
 (b\otimes c)^{X\times J}=B\otimes c.
\end{equation}
Choose a common splitting $p$-modular system $(K,\mathcal O,k)$,
and let $\lambda_d$ denote the central character over $k$ of a block
$d$. For $Y\leq Z$, let $\pi_Y^Z\colon kZ\longrightarrow kY$
be the linear projection that discards the terms supported outside $Y$.
The assumption $b^X=B$ means that
\[
 \lambda_b\bigl(\pi_L^X(z)\bigr)=\lambda_B(z)
 \qquad(z\in Z(kX)).
\]
Every conjugacy class of $X\times J$ is the product of a conjugacy
class of $X$ and a conjugacy class of $J$. On the resulting basis of
class sums,
\[
 \pi_{L\times J}^{X\times J}=\pi_L^X\otimes\operatorname{id}_{kJ},
 \qquad \lambda_{b\otimes c}=\lambda_b\otimes\lambda_c.
\]
The induced central character is therefore
$\lambda_B\otimes\lambda_c=\lambda_{B\otimes c}$.
Hence the block induction in \eqref{r8:eq:4} is defined and has
the stated value. Moreover, if $Q$ is a defect group of $c$, then
$D\times Q$ is a defect group of $b\otimes c$, and
\[
 N_{X\times J}(D\times Q)
 =N_X(D)\times N_J(Q)\leq L\times J.
\]

Finally, the blocks of the prescribed characters have the common
specified defect group $D\times1$, and
\[
 N_{X\times1}(D\times1)=N_X(D)\times1\leq L\times1.
\]
Thus all the conditions of \cite[Definition~1.1]{ChenCT} hold for
the pair $\mathcal P,\mathcal P'$. Transporting the relation
through \eqref{r8:eq:2} proves \eqref{r8:eq:1}.
\end{proof}

\begin{corollary}\label{r8:2}
Let $S$ be a non-abelian finite simple group with trivial Schur
multiplier and trivial outer automorphism group. Suppose that, for
every nontrivial radical $p$-subgroup $R$ of $S$, there is an
$\mathcal H$-equivariant bijection
\[
 \omega_R\colon\Irr_0(S\mid R)\longrightarrow
                 \Irr_0(N_S(R)\mid R)
\]
which maps the height-zero characters of each block onto those of its
Brauer correspondent. Then $S$ satisfies Definition~\ref{r6:1} for
$p$.
\end{corollary}

\begin{proof}
The universal covering group is $S$ itself. For each $R$, take
$L_R=N_S(R)$. This subgroup is proper: otherwise $R$ would be a nontrivial
normal $p$-subgroup of the non-abelian simple group $S$.
The group $\Gamma=\Aut(S)_R=\Inn(S)_R$ acts by inner automorphisms
on both $S$ and $L_R$. It therefore stabilizes $L_R$ and fixes every character in both sets. Hence $\omega_R$ is
$\Gamma\times\mathcal H$-equivariant. Central characters are
preserved because $Z(S)=1$.

Let $\varphi=\omega_R(\vartheta)$. Equivariance and injectivity give
$\mathcal H_\vartheta=\mathcal H_\varphi$. Also,
$\Gamma_{\vartheta^{\mathcal H}}=\Gamma$, and the subgroup $W$
in Definition~\ref{r6:1} is trivial. Apply Lemma~\ref{r8:1} with
$X=S$, $D=R$, and $L=L_R$. The resulting relation is precisely \eqref{r6:eq:1}, with the
required block condition on every intermediate group. For the central radical subgroup $R=1$, use the identity
correspondence stipulated in Definition~\ref{r6:1}.
\end{proof}

\subsection{An application to cyclic defect groups}\label{r8:sub:2}

\begin{theorem}\label{r8:3}
Let $S$ be a non-abelian finite simple group with trivial Schur
multiplier and trivial outer automorphism group. If $|S|_p=p$, then
$S$ satisfies the inductive condition in Definition~\ref{r6:1} for
$p$.
\end{theorem}

\begin{proof}
Let $P\in\Syl_p(S)$ and set $N=N_S(P)$. Then $|P|=p$.
Every nontrivial $p$-subgroup of $S$ is a Sylow subgroup. Such
a subgroup is radical, since $O_p(N_S(P))=P$. We construct the
correspondence required by Corollary~\ref{r8:2} for $P$; the same
construction applies to each conjugate of $P$.

Let $\mathcal B_P$ be the set of blocks of $S$ having the specified
defect group $P$. By \cite[Lemma~1.4]{ChenCT}, this set is $\mathcal H$-stable,
and Brauer correspondence is $\mathcal H$-equivariant. For $B\in\mathcal B_P$, let $b_B$ denote its Brauer
correspondent in $N$. Since these blocks have cyclic defect groups, all their
irreducible characters have height zero.
By \cite[Theorem~3.4(a)]{Nav04}, there is an
$\mathcal H_B$-equivariant bijection
\[
 F_B\colon\Irr_0(B)\longrightarrow\Irr_0(b_B).
\]
The cited theorem is formulated over a finite cyclotomic field.
Taking such a field to contain the character values of $S$ and $N$
and inflating the Galois action to $\mathcal H$ gives the stated
equivariance.

Choose one representative $B$ from each $\mathcal H$-orbit on
$\mathcal B_P$, together with a map $F_B$ as above. For $\sigma\in\mathcal H$ and
$\chi\in\Irr_0(B)$, define
\begin{equation}\label{r8:eq:5}
 \omega_P(\chi^\sigma)=F_B(\chi)^\sigma.
\end{equation}
To verify that this is well-defined, suppose that
$\chi^\sigma=\widetilde\chi^{\widetilde\sigma}$ with
$\chi,\widetilde\chi\in\Irr_0(B)$. Then
$\rho=\sigma\widetilde\sigma^{-1}\in\mathcal H_B$ and
$\widetilde\chi=\chi^\rho$. By the $\mathcal H_B$-equivariance of
$F_B$,
\[
 F_B(\widetilde\chi)^{\widetilde\sigma}
 =F_B(\chi)^{\rho\widetilde\sigma}
 =F_B(\chi)^\sigma.
\]
The chosen representatives lie in disjoint block orbits, which
excludes any further ambiguity.

Thus \eqref{r8:eq:5} defines an $\mathcal H$-equivariant bijection
on the union of the character sets of the relevant blocks. Brauer's first main
theorem gives a bijection between the blocks of $S$ having defect group
$P$ and the blocks of $N$ having defect group $P$. Thus the constructed map is a
bijection
\[
 \omega_P\colon\Irr_0(S\mid P)\longrightarrow\Irr_0(N\mid P)
\]
which maps each block onto its Brauer correspondent.
Corollary~\ref{r8:2} now applies.
\end{proof}

\subsection{Sporadic groups}\label{r8:sub:3}

\begin{corollary}\label{r8:4}
The inductive condition in Definition~\ref{r6:1} holds for each pair
$(S,p)$ in the following table:
\[
 \begin{array}{c|l}
 S & p\\ \hline
 M_{11} & 5,\ 11\\
 J_1    & 3,\ 5,\ 7,\ 11,\ 19\\
 M_{23} & 5,\ 7,\ 11,\ 23
 \end{array}
\]
\end{corollary}

\begin{proof}
All three groups have trivial Schur multiplier and trivial outer
automorphism group. Their orders are
\[
 \begin{aligned}
 |M_{11}|&=2^4\cdot3^2\cdot5\cdot11,\\
 |J_1|&=2^3\cdot3\cdot5\cdot7\cdot11\cdot19,\\
 |M_{23}|&=2^7\cdot3^2\cdot5\cdot7\cdot11\cdot23.
 \end{aligned}
\]
These data are recorded in \cite{Atlas} and in the individual ATLAS
entries \cite{AtlasM11,AtlasJ1,AtlasM23}. For each prime in the
corresponding row, $|S|_p=p$. The result follows from
Theorem~\ref{r8:3}.
\end{proof}

In these examples, the cyclic-defect correspondence supplies the
blockwise Galois-equivariant bijection, while Lemma~\ref{r8:1}
provides the associated representations and all the block relations
required by the inductive condition. The argument above does not
address the omitted primes or groups with nontrivial outer
automorphism group or Schur multiplier. The next subsection uses a
different argument for two families with nontrivial outer automorphism
group.

\subsection{Suzuki and small Ree groups in defining characteristic}\label{r8:sub:4}

The defining-characteristic McKay--Navarro correspondence for Suzuki
and small Ree groups can be strengthened to the block relation
required in Definition~\ref{r6:1}. We verify the additional block
condition using the associated representations from that
correspondence and the trace criterion of Navarro and Sp\"ath.

\begin{theorem}\label{r8:5}
The following simple groups satisfy the iAMN condition of
Definition~\ref{r6:1} in their defining characteristic:
\begin{enumerate}
\item $S={}^2B_2(2^{2m+1})$, with $m\geq2$, for $p=2$;
\item $S={}^2G_2(3^{2m+1})$, with $m\geq1$, for $p=3$.
\end{enumerate}
In each case, the universal covering group is $S$ itself. For every nontrivial
radical $p$-subgroup $R$ of $S$, the local subgroup can be chosen
to be $L_R=N_S(R)$.
\end{theorem}

\begin{proof}
Write $q=p^{2m+1}$ and fix a Sylow $p$-subgroup $P$ of $S$.
Set $B=N_S(P)$ and $\Gamma=\Aut(S)_P$. In the stated ranges,
$S$ has trivial center and trivial Schur multiplier. Moreover,
\[
 B=P\rtimes T,\qquad
 \Aut(S)=S\rtimes D,\qquad
 \Gamma=\Inn(B)\rtimes D,
\]
where $T$ is cyclic of order $q-1$ and $D=\langle F_p\rangle$
is the cyclic group of field automorphisms, of order $2m+1$.
We identify $S$ with its inner automorphism group when it occurs
inside $\Aut(S)$. These structural facts and the choice of a
$D$-stable Sylow subgroup are recorded in
\cite[Section~3.1]{Joh22}.

We first prove that
\begin{equation}\label{r8:eq:6}
 C_{\Aut(S)}(P)=Z(P).
\end{equation}
For a Suzuki group, use the coordinates
\[
 P=\{(a,b):a,b\in\mathbb F_q\},\qquad
 Z(P)=\{(0,b):b\in\mathbb F_q\}.
\]
Set $\theta=2^{m+1}$. For $\lambda\in\mathbb F_q^\times$,
the torus element $h_\lambda$ acts by
\[
 (a,b)^{h_\lambda}=(\lambda a,\lambda^{\theta+1}b).
\]
These coordinates and the multiplication in $P$ are given in
\cite[Section~4.3.3, (4.26)--(4.27)]{GKLL08}.
An element centralizing $P$ normalizes $P$ and hence has the form
$uh_\lambda F_2^k$, with $u\in P$ and $0\leq k<2m+1$.
Inner automorphisms induced by $P$ act trivially on $P/Z(P)$,
so centralization forces
\[
 \lambda a^{2^k}=a\qquad(a\in\mathbb F_q).
\]
Taking $a=1$ gives $\lambda=1$, and varying $a$ then gives $k=0$.
The remaining element $u$ centralizes $P$ precisely when
$u\in Z(P)$.

For a small Ree group, set $\theta=3^{m+1}$ and use the
coordinates
\[
 P=\{x(a,b,c):a,b,c\in\mathbb F_q\}
\]
from \cite[Section~4.3.4, (4.31)--(4.32)]{GKLL08}.
The first coordinate is additive under multiplication, and
\[
 x(a,b,c)^{h_\lambda}
 =x(\lambda^{2-\theta}a,\lambda^{\theta-1}b,\lambda c).
\]
Thus inner automorphisms induced by $P$ leave the first coordinate
unchanged. If $uh_\lambda F_3^k$ centralizes $P$, with
$u\in P$ and $0\leq k<2m+1$, then
\[
 \lambda^{2-\theta}a^{3^k}=a\qquad(a\in\mathbb F_q).
\]
Taking $a=1$ and then varying $a$ gives
$\lambda^{2-\theta}=1$ and $k=0$. Since $\theta^2=3q$,
any common divisor of $\theta-2$ and $q-1$ divides $4-3=1$.
Hence $\lambda=1$, and again $u\in Z(P)$. This proves
\eqref{r8:eq:6} in both cases.

We next identify the relevant character sets and blocks.
Fix a common splitting $p$-modular system $(K,\mathcal O,k)$
for the groups and representations considered below.
Since $C_S(P)=Z(P)$ is a $p$-group, the principal block is the
unique block of $S$ with defect group $P$. Indeed, the Brauer
images at $P$ of distinct block idempotents of full defect are
nonzero orthogonal idempotents in the local algebra
$kC_S(P)=kZ(P)$. Thus there is at most one such block, and the
principal block has full defect. The same argument applies to
$B$. Since $P\trianglelefteq B$ is a Sylow subgroup, every block
of $B$ has defect group $P$, so $B$ has only its principal block.
Consequently,
\begin{equation}\label{r8:eq:7}
 \begin{aligned}
 \Irr_0(S\mid P)&=\Irr_0(B_0(S))=\Irr_{p'}(S),\\
 \Irr_0(B\mid P)&=\Irr_0(B_0(B))=\Irr_{p'}(B),
 \end{aligned}
\end{equation}
where $B_0(Y)$ denotes the principal $p$-block of a group $Y$.
The two principal blocks in \eqref{r8:eq:7} are Brauer
correspondents.

By \cite[Proposition~3.4 and Theorem~3.8]{Joh22}, there is a
$\Gamma\times\mathcal H$-equivariant bijection
\[
 \Omega\colon\Irr_{p'}(S)\longrightarrow\Irr_{p'}(B).
\]
Fix $\chi\in\Irr_{p'}(S)$ and set $\psi=\Omega(\chi)$ and
\[
 \Gamma_0=\Gamma_{\chi^{\mathcal H}},\qquad
 E=S\rtimes\Gamma_0,\qquad F=B\rtimes\Gamma_0.
\]
Equivariance gives $\Gamma_\chi=\Gamma_\psi$ and
$\Gamma_{\chi^{\mathcal H}}=\Gamma_{\psi^{\mathcal H}}$.
The construction in \cite[Remark~2.3(c), Propositions~3.6--3.7,
and the proof of Theorem~3.8]{Joh22} provides representations
$\mathcal P$ of $E_\chi$ and $\mathcal P'$ of $F_\psi$
realizing
\[
 (E,S,\chi)_{\mathcal H}\geq_c(F,B,\psi)_{\mathcal H}.
\]
Both factor sets are trivial, and the mixed comparison functions
for these representations agree; in this construction they are
all trivial.

We recall a property of these particular representations that
will be used to verify the block condition. They are obtained
by first extending $\chi$ and $\psi$ to $S\rtimes D_\chi$
and $B\rtimes D_\chi$, respectively, and then applying the
canonical extension over inner automorphisms. Let
\[
 \pi\colon E\longrightarrow\Aut(S),\qquad
 (s,\gamma)\longmapsto c_s\gamma
\]
be the natural action homomorphism, where $c_s$ denotes
conjugation by $s$. Its image is $S\rtimes D_{\chi^{\mathcal H}}$,
and
\[
 \ker\pi=\{(b^{-1},c_b):b\in B\}=C_E(S)\leq F_\psi.
\]
The restrictions of $\pi$ to $E_\chi$ and $F_\psi$ have
images $S\rtimes D_\chi$ and $B\rtimes D_\chi$, respectively.
The representations $\mathcal P$ and $\mathcal P'$ are
the pullbacks of the extensions just described. In particular,
\begin{equation}\label{r8:eq:8}
 \mathcal P(z)=I_{\chi(1)},\qquad
 \mathcal P'(z)=I_{\psi(1)}\qquad(z\in\ker\pi).
\end{equation}
The kernel need not be central in $E$; the argument uses its
trivial action on both representations.

We now apply \cite[Theorem~4.4]{NS14} to the ordinary character
triples $(E_\chi,S,\chi)$ and $(F_\psi,B,\psi)$, using the same
pair $(\mathcal P,\mathcal P')$. Their blocks have the common
defect group $P$, as established above, and $N_S(P)=B$.
The common factor set and central scalar conditions already hold.
It remains to check the normalized trace equality for every
$p'$-element $x\in F_\psi$ such that
$P\in\Syl_p(C_S(x))$. Such an element centralizes $P$, so
\[
 \pi(x)\in C_{\Aut(S)}(P)=Z(P).
\]
Since $\pi(x)$ has $p'$-order, we have $\pi(x)=1$.
Equation~\eqref{r8:eq:8} therefore gives
\[
 \tr\mathcal P(x)=\chi(1),\qquad
 \tr\mathcal P'(x)=\psi(1).
\]
Both characters have height zero and $p'$-degree. Thus the
required trace equality is
\[
 \left(\frac{|S|_{p'}\tr\mathcal P(x)}{\chi(1)_{p'}}\right)^*
 =(|S|_{p'})^*
 =(|B|_{p'})^*
 =\left(\frac{|B|_{p'}\tr\mathcal P'(x)}{\psi(1)_{p'}}\right)^*.
\]
Here $*$ denotes reduction modulo $J(\mathcal O)$.
The middle equality follows from the Sylow congruence
$|S:B|\equiv1\pmod p$, since $|S|_p=|B|_p$.
The trace criterion shows that, for every $S\leq J\leq E_\chi$,
the tensor correspondence afforded by
$(\mathcal P,\mathcal P')$ satisfies
\[
 \bl(\tau_J(\xi))^J=\bl(\xi)
 \qquad\bigl(\xi\in\Irr(J\mid\chi)\bigr),
\]
with the indicated block induction defined.
This is the intermediate-group condition in
Definition~\ref{red:block}\textup{(iv)}. Together with the
$\mathcal H$-triple conditions already verified for the same pair,
it proves
\[
 (S\rtimes\Gamma_{\chi^{\mathcal H}},S,\chi)_{\mathcal H}
 \geq_b
 (B\rtimes\Gamma_{\chi^{\mathcal H}},B,\psi)_{\mathcal H}.
\]
Since $Z(S)=1$, the central-character and central-kernel
requirements of Definition~\ref{r6:1} involve no further quotient.
This proves the condition for $R=P$, and the same argument
applies to every conjugate Sylow subgroup.

Finally, the defining-characteristic block theorem states that
every block of $S$ of positive defect has full defect; see
\cite[Theorem~8.5 and Section~20.3]{Hum05}. Hence, if $R$ is a
nontrivial radical $p$-subgroup that is not a Sylow subgroup,
then $\Irr_0(S\mid R)$ is empty. Brauer's first main theorem
implies that $\Irr_0(N_S(R)\mid R)$ is also empty.
Take $L_R=N_S(R)$ and the empty bijection. This subgroup is
$\Aut(S)_R$-stable and proper in $S$, since $S$ is simple and
$R$ is nontrivial. All conditions involving a pair of characters
are then vacuous. For the central radical subgroup $R=1$, take
$L_R=S$ and the identity correspondence, as prescribed in
Definition~\ref{r6:1}. This completes the verification for every
radical $p$-subgroup.
\end{proof}

The theorem does not include ${}^2B_2(8)$, whose nontrivial
Schur multiplier requires a separate verification on its universal
covering group, or the exceptional groups obtained from the
smallest nonsimple Ree groups.

\Needspace{10\baselineskip}

\end{document}